\documentclass[11pt, twoside]{amsart}
\usepackage[top=4cm, bottom=3cm, left=3cm, right=3cm, a4paper]{geometry}

\usepackage{graphicx}
\usepackage{amssymb}
\usepackage{epstopdf}
\usepackage{amsrefs}
\title{Iterated Hopf Ore Extensions over Group Rings}
\author{Can Hat\.{i}po\u{g}lu}
\address{College of Engineering and Technology, American University of the Middle East, Kuwait}
\email{osman.hatipoglu@aum.edu.kw}
\author{Christian Lomp}
\address{Department of Mathematics of the Faculty of Science and Center of Mathematics, University of Porto, Rua Campo Alegre, 687, 4169-007 Porto, Portugal}
\email{clomp@fc.up.pt}
\newtheorem{theorem}{Theorem}[section]
\newtheorem{lemma}[theorem]{Lemma}
\newtheorem{corollary}[theorem]{Corollary}
\newtheorem{proposition}[theorem]{Proposition}
\newtheorem{example}[theorem]{Example}
\newtheorem{remark}[theorem]{Remark}
\theoremstyle{definition}

\newcommand{\ZZ}{\mathbb{Z}}
\newcommand{\KK}{\mathbb{K}}

\newcommand{\DD}{\mathbb{D}}

\newcommand{\ou}{\overline{u}}
\newcommand{\ov}{\overline{v}}

\newcommand{\Irr}{\operatorname{Irr}}

\newcommand{\wind}[2]{[#1]_{#2}^{\sigma}}

\begin{document}
\begin{abstract}
    We introduce and study a class of Hopf algebras $H(G, \chi, \eta, b, c, \beta)$ which are two-step Ore extensions of a group algebra $\KK[G]$. This construction unifies and generalizes some known families of Hopf algebras such as generalized Taft algebras and Hopf algebras related to $\mathfrak{sl}_2$ constructed by Wang, Wu, and Tan. We analyze the ring theoretical properties of these algebras and classify all finite dimensional simple modules over them. We also consider the tensor products of simple modules in the zero derivation case. 
\end{abstract}

\keywords{ Hopf-Ore extension; Hopf algebras, irreducible representations;  16T05; 16P40; 16S15}
\thanks{The second author was partially supported by Centro de Matemática da Universidade do Porto, member of LASI, which is financed by
national funds through FCT – Fundação para a Ciência e a Tecnologia, I.P., under the project with reference UID/00144/2025, doi: https://doi.org/10.54499/UID/00144/2025. He would like to dedicate this paper to the cherished memory of his father, Gerhard Lomp.}
\maketitle

\section{Introduction}
Hopf algebras that are constructed as quotients of iterated Ore extensions have played an important role in the development of modern Hopf algebra theory. For example, Kaplansky's tenth conjecture (see \cite{Kaplansky}), which states that in each dimension $n$ there are (up to isomorphism) only finitely many $n$-dimensional Hopf algebras, has been shown to fail by constructing an infinite family of non-isomorphic quotients of Hopf Ore extensions with the same dimension (see \cites{Beattie, Gelaki, Muller}). 

The aim of this paper is to study the finite dimensional representation theory of a broad class of Hopf algebras arising as quotients of two-step iterated Hopf Ore extensions of group algebras. More precisely, we consider two-step Ore extensions of a group algebra $\KK[G]$ obtained by adjoining two skew-primitive elements with the commutation laws determined by some characters $\chi, \eta$ of $G$. Our results generalize those obtained in \cite{Wang-Wu-Tan}. The class of generalized Taft algebras (see \cite{Halbig}) also fits into our construction.

A central theme of the paper is a duality determined by the nature of the second Ore extension. When the derivation is zero, after a suitable change of variables the algebra reduces to a skew group ring $\KK[x,y]\#G$, and the representation theory can be analyzed using induction techniques for skew group algebras. In contrast, when the derivation is nonzero, the resulting algebra behaves more like a noncommutative differential operator ring. In this latter case, the defining relations force a strong rigidity condition: the characters governing the two Ore extensions must be inverses of each other. This constraint has significant consequences for the structure of finite dimensional simple modules and leads to explicit families of simple modules that do not occur in the skew group ring setting.

The paper is organized as follows. In Section 2 we construct the Hopf algebras $H = H(G, \chi, \eta, b, c,\beta)$ and record their basic structural properties. Section 3 discusses examples and connects our framework to existing constructions in the literature. Section 4 studies ring-theoretic and homological properties, including Hopf algebra quotients. Section 5 is devoted to the classification of finite dimensional simple modules, which is divided into the skew group ring case and the differential operator case. In the skew group ring case, simple modules are described via induction from one-dimensional modules over appropriate subalgebras, recovering and extending classical results. In the nonzero derivation case, we construct and classify torsion, torsion-free, and mixed simple modules, give precise isomorphism criteria, and show that all finite dimensional simple modules arise from these constructions. Finally, we close in Section 6 with the tensor products of simple modules in the zero derivation case.

\section{The Construction}

Throughout the text, we will assume that $\KK$ is an algebraically closed field of characteristic zero and all homomorphisms are $\KK$-linear. If $R$ is a $\KK$-Hopf algebra, we will denote its comultiplication by $\Delta$, its counit by $\epsilon$, and its antipode by $S$. Usually, we will use the simplified Sweedler notation $\Delta(h) = h_1 \otimes h_2$ to indicate the comultiplication of an element $h \in R$. A \emph{group-like} element is a nonzero element $g \in R$ which satisfies $\Delta(g) = g \otimes g$. The set of group-like elements $G(R)$ forms a group under the multiplication of $R$, where the inverse of $g \in G(R)$ is given by $S(g)$. Given $g, h \in G(R)$, we say that an element $x \in R$ is $(g,h)$-\emph{primitive} (resp. $(g,h)$-\emph{skew-primitive}) if $\Delta(x) = x \otimes g + h\otimes x$ (resp. $\Delta(x) = g \otimes x + x \otimes h$). Following \cite{Brown-O'Hagan-Zhang}, we will call a homomorphism of $\KK$-algebras $\chi:R \to \KK$ a \emph{character}. The set $\Xi(R)$ of characters of $R$ forms a group with the convolution product as the multiplication. Given $\chi \in \Xi(R)$, the map $\tau_\chi : R\to R$ defined by $\tau_\chi (r) = \chi(r_1)r_2$ for $r \in R$ is a $\KK$-automorphism of $R$, called a \emph{(left) winding automorphism} of $R$.

Recall that given an associative, unital $\KK$-algebra $R$, an automorphism $\sigma$, and a $\sigma$-derivation $\delta$ of $R$, the \emph{Ore extension} $R[x; \sigma, \delta]$ is defined as the free left $R$-module with basis $\{x^i \mid i \geq 0\}$ subject to
\[xr = \sigma(r)x + \delta(r)\]
for all $r \in R$. In case $R$ is a $\KK$-Hopf algebra, Panov's Theorem \cite{Panov} shows that $R[x; \sigma, \delta]$ is a Hopf algebra with $R$ as a Hopf subalgebra and $x$ being $(1,g)$-primitive such that 
\[\Delta(x) = x\otimes 1 + g\otimes x, \epsilon(x) = 0, S(x) = -g^{-1}x\]
for some group-like element $g \in G(R)$ if and only if there exists $\chi \in \Xi(R)$ with
\begin{itemize}
    \item[(i)] $\sigma = \tau_\chi$ and $\sigma(r) = g\chi(r_2)r_1 g^{-1}$ for all $r \in R$.
    \item[(ii)] $\Delta(\delta(r)) = \delta(r_1) \otimes r_2 + gr_1 \otimes \delta(r_2)$.
\end{itemize}
In this case, $R[x;\sigma, \delta]$ is called a \emph{Hopf Ore extension (HOE)} of $R$ (see \cite{Panov} or \cite[Theorem 2.4]{Brown-O'Hagan-Zhang}). If $G$ is a group, we let $Z(G)$ denote its center and $\widehat{G}$ denote its character group, that is, the group of homomorphisms $G \to \KK^{\times}$. 

The following theorem introduces the central object of our study:
\begin{theorem}[The Hopf algebras $H(G, \chi, \eta, b, c, \beta)$]\label{theorem:introducing H}
    For a group $G$ and pairs $(\chi, b), (\eta, c) \in \widehat{G} \times Z(G)$ such that $\eta(b) = \chi(c)^{-1}$ and $\beta \in \KK$, there exists a $\tau_\eta$-derivation $\delta$ such that the iterated Ore extension $H = \KK[G][x; \tau_\chi][y; \tau_\eta, \delta]$ is a Hopf algebra such that $x$ is $(1, b)$-primitive and $y$ is $(1,c)$-primitive. That is, 
    \[\Delta(x) = x \otimes 1 + b \otimes x,\quad \Delta(y) = y \otimes 1 + c \otimes y.\]
    In particular, $H$ is generated as an algebra by the elements $g$ of $G$ and $x$ and $y$ such that $xg = \chi(g)gx$ and $yg = \eta(g)gy$ for all $g \in G$ and 
    \[yx = qxy + \beta(1 - cb)\]
    where $q = \eta(b)$. We shall denote this Hopf algebra by $H(G, \chi, \eta, b, c, \beta)$.
\end{theorem}
\begin{proof}
    By \cite[Proposition 2.2]{Panov}, any Hopf Ore extension $R[x; \sigma]$ of a group ring $R = \KK[G]$ is determined by a pair $(\chi, b) \in \widehat{G} \times Z(G)$ of a group character $\chi$ and a central element $b \in Z(G)$. Then $\sigma$ is determined by $\sigma(g) = \chi(g)g$ for all $g \in G$ and 
    \[xg = \chi(g)gx,\quad \Delta(x) = x \otimes 1 + b \otimes x,\quad \epsilon(x) = 0,\quad S(x) = -b^{-1}x.\]
    Hence, $A = R[x; \tau_\chi]$ is a Hopf Ore extension with $x$ being $(1,b)$-primitive and $R = \KK[G]$. Let $(\eta, c) \in \widehat{G} \times Z(G)$ and extend $\eta$ to an algebra homomorphism $\eta: A \to \KK$ by setting $\eta(x) = 0$. Let $\tau_\eta$ be the left winding automorphism of $A$ defined by $\tau_\eta(a) = \eta(a_1)a_2$ for all $a \in A$. In particular,
    \[\tau_\eta(x) = \eta(x)\cdot 1 + \eta(b) \cdot x = \eta(b)x.\]
    Moreover, 
    \[c(\eta(x_2)x_1)c^{-1} = c(\eta(1)x + \eta(x)b)c^{-1} = \chi(c)^{-1}x\]
    Thus, the first property of Panov's Theorem, that is, $\tau_\eta(x) = c(\eta(x_2)x_1)c^{-1} $ is satisfied if and only if $\eta(b) = \chi(c)^{-1}$. Now, we wish to define a $\tau_\eta$-derivation $\delta$ of $R[x; \tau_\chi]$ such that $\delta|_R = 0$ and $\delta(x) \in R^{+} = \ker\epsilon$. The second condition in Panov's Theorem says that we need
    \[\Delta(\delta(x)) = \delta(x) \otimes 1 + \delta(b) \otimes x + cx \otimes \delta(1) + cb \otimes \delta(x) = \delta(x) \otimes 1 + cb \otimes \delta(x).\]
    This means that $\delta(x)$ must be a $(1, cb)$-primitive element of $R$, which is fulfilled by setting $\delta(x) = \beta(1 - cb)$ for some $\beta \in \KK$.
\end{proof}
\begin{remark}\label{remark:chi-eta-inverse}
    As a natural consequence of the construction, in the case $\beta \neq 0$ and $cb \neq 1$, the characters $\chi$ and $\eta$ are strongly connected. For simplicity, let us write $e = \beta(1 - cb)$. For any $g \in G$, conjugating the relation $yx - qxy = e$ with $g$ (or $g^{-1}$) yields
    \[g^{-1}(yx - qxy)g = g^{-1}yxg - q g^{-1}xyg = g^{-1}eg = e.\]
    Expanding the terms using the relations $g^{-1}yg = \eta(g)y$ and $g^{-1}xg = \chi(g)x$, we get
    \[ (g^{-1}yg)(g^{-1}xg) - q(g^{-1}xg)(g^{-1}yg) = \eta(g)\chi(g)yx - q\chi(g)\eta(g)xy. \]
    Factoring out the scalar $\chi(g)\eta(g)$, this becomes
    \[ \chi(g)\eta(g)(yx - qxy) = \chi(g)\eta(g)e. \]
    Comparing this with the conjugated relation $e$, it follows that $(\chi(g)\eta(g) - 1)e = 0$. Since we assume $e \neq 0$, we must have $\chi(g)\eta(g) = 1$, or $\chi(g) = \eta(g)^{-1}$. Since this holds for every $g \in G$, the characters $\chi$ and $\eta$ must be inverses of each other.
\end{remark}
After this remark, in the nonzero derivation case, we will always take $\eta = \chi^{-1}$. Also, in this case, the assumption in Theorem~\ref{theorem:introducing H} that $\eta(b) = \chi(c)^{-1}$ is true if and only if $\chi(b) = \chi(c)$.

\section{Examples}\label{section:examples}
In this section, we provide a list of examples that fit into our framework. 
\begin{example}[Takeuchi's Hopf Algebra $U(1)$]
    In \cite{Takeuchi}, Takeuchi describes a family of Hopf algebras $U(n)$. In particular, $U(1)$ is generated by $c, x$, and $y$ such that $c$ is invertible and 
    \[xc = -cx, \quad yc = -cy, \quad yx = xy + c^{-1} - c\]
    while $c$ is group-like, $x$ is $(c, 1)$-primitive, and $y$ is $(1, c^{-1})$-primitive. After a change of variable $x' = xc^{-1}$, we get that $x'$ is $(1, c^{-1})$-primitive and 
    \[yx' = yxc^{-1} = xyc^{-1} + (c^{-2} - 1) = -x'y + (c^{-2} - 1).\]
    Thus $U(1) = H(\langle b \rangle, \chi, \chi, b, b, -1)$, where $b = c^{-1}$ and $\chi$ is defined by $\chi(b) = -1$.  
\end{example}
\begin{example}[Generalized Taft Algebras]
    In \cite{Halbig}, certain generalized Taft algebras have been constructed (see also \cite{Halbig-Krahmer}), which are defined as algebras of the form
    \[H = \left(\mathcal{B}(V)\# \KK[G] \right)^{cop}\]
    where $V$ is a Yetter-Drinfeld module of type $D_2$ over a group ring $\KK[G]$ of a finite cyclic group $G$, $\mathcal{B}(V)$ is the so-called Nichols algebra of $V$ and $\mathcal{B}(V) \# \KK[G]$ the so-called bosonization. Expressed in terms of generators and relations, one obtains that $\left(\mathcal{B}(V)\# \KK[G] \right)^{cop}$ is a quotient of a Hopf algebra of our type. More precisely, by \cite{Halbig}*{3.14}, $H$ is generated by $g, x$, and $y$ subject to the relations
    \[g^N = 1, \quad gy = q^{a_{22}}yg, \quad gx = q^{a_{12}}xg, \quad xy = q^{a_{11} a_{22}}yx\]
    as well as $x^{N_x} = 0 = y^{N_y}$, where $q$ is a primitive $N$th root of unity, $(a_{ij}) \in M_{2}(\ZZ_N)$ satisfies the properties
    \[a_{11} \neq a_{12}\ (\mathrm{mod }N), \quad a_{21} \neq a_{22}\ (\mathrm{mod } N),\quad a_{11}a_{22} + a_{12}a_{21} = 0\ (\mathrm{mod} N), \]
    while $N_y$ is the multiplicative order of $a_{11}a_{12}$ and $N_x$ that of $a_{21}a_{22}$ in $\ZZ_N$. Furthermore, $g$ is group-like, $x$ is $(g^{a_{11}}, 1)$-primitive and $y$ is $(g^{a_{21}}, 1)$-primitive.

    Let $G = \ZZ_N = \langle g \rangle$ be the cyclic group of order $N$ and define characters $\chi, \eta \in \widehat{G}$ by $\chi(g) = q^{-a_{22}}$ and $\eta(g) = q^{-a_{12}}$. Set $b = g^{-a_{21}}$ and $c = g^{-a_{11}}$. Then, by the hypothesis on $(a_{ij})$, we have 
    \[\chi(c^{-1}) = \chi(g^{a_{11}}) = q^{-a_{11}a_{22}} = q^{a_{12}a_{21}} = \eta(b).\]
    Note also that $\chi(c) = \chi(g)^{-a_{11}} = q^{a_{11}a_{22}}$. Let $H' = H(\langle g \rangle, \chi, \eta, b, c, 0)$ and denote its generators by $X$ and $Y$. Set $y' = Xb^{-1}$ and $x' = c^{-1}Y$. Then,
    \[x' y' = c^{-1} YXb^{-1} = \eta(b)c^{-1}XYb^{-1} = \eta(b)\chi(c)\eta(b^{-1})y' x' = \chi(c)y' x' = q^{a_{11}a_{22}}y'x'.\]
    Furthermore, $\Delta(y') = \Delta(X)\Delta(b^{-1}) = Xb^{-1} \otimes b^{-1} + 1 \otimes Xb^{-1} = y' \otimes g^{a_{21}} + 1 \otimes y'$ shows that $y'$ is $(g^{a_{21}}, 1)$-primitive. Similarly, $x'$ is $(g^{a_{11}}, 1)$-primitive. Thus, mapping $x' \mapsto x$ and $y' \mapsto y$ yields an isomorphism of Hopf algebras $H'/ \langle(x')^{N_x}, (y')^{N_y}\rangle \simeq H$. 
\end{example}
The following class of examples is one of our main motivations to introduce the Hopf algebras $H(G, \chi, \eta, b, c, \beta)$.
\begin{example}\label{example_sl2}
    Tan et al. studied in \cite{Wang-Wu-Tan} a class of pointed Hopf algebras which fit into our setting. We briefly revise their construction and show that after a change of variables, it is equal to the type of Hopf algebras constructed in the previous sections. Let $q \in \KK$ be a primitive $n$th root of unity, $1 \leq n_1 \leq n$ with $2 \nmid \gcd(n, n_1)$ and $\beta_1, \beta_2, \beta_3 \in \KK$. The Hopf algebra $H_\beta$ considered in \cite{Wang-Wu-Tan} is generated by $a, b, c, X, Y$ subject to 
    \[Xa = qaX,\quad Ya = q^{-1}aY,\quad YX = q^{-n_1}XY + \beta_3 (a^{2n_1} - bc),\]
    with $b$ and $c$ central and $X^n = \beta_1 (a^{nn_1} - b^n)$ and $Y^n = \beta_2 (a^{nn_1} - c^n)$. The coalgebra structure of $H_\beta$ is given by 
    \[\Delta(X) = X \otimes a^{n_1} + b \otimes X,\quad \Delta(Y) = Y \otimes a^{n_1} + c \otimes Y\]
    and $a, b, c$ being group-like elements.

    Let $G = \ZZ^3 = \langle a, e, f\rangle$ be the free abelian group in three generators $a, e, f$. Define the character $\chi \in \widehat{G}$ by $\chi(a) = q$, $\chi(e) = \chi(f) = q^{-n_1}$ and consider the Hopf algebra
    \[H = H(G, \chi, \chi^{-1}, e, f, \beta_3),\]
    which, as an algebra, is generated by $a, e, f, x, y$ subject to 
    \[xa = qax,\quad ya = q^{-1}ay,\quad yx = q^{n_1}xy + \beta_3 (1 - ef),\]
    as well as $xu = q^{-n_1}ux$, $yu = q^{n_1}uy$ for $u \in \{e, f\}$. The comultiplications of $x$ and $y$ are given by 
    \[\Delta(x) = x \otimes 1 + e \otimes x,\quad \Delta(y) = y \otimes 1 + f \otimes y\]
    and $a, e, f$ are group-like elements. Set $b:= ea^{n_1}$ and $c:= fa^{n_1}$. Then $\chi(b) = \chi(c) = 1$ shows that $b$ and $c$ commute with $x$ and $y$. Set $X:= xa^{n_1}$ and $Y:= a^{n_1}y$. Then,
    \[YX = a^{n_1}yxa^{n_1} = a^{n_1}(q^{n_1}xy + \beta_3 (1 - ef))a^{n_1} = q^{-n_1}XY + \beta_3 (a^{2n_1} - cb).\]
    Moreover, $\Delta(X) = \Delta(x)\Delta(a^{n_1}) = X \otimes a^{n_1} + b \otimes X$. Analogously, $\Delta(Y) = Y \otimes a^{n_1} + c \otimes Y$ holds. Since $H$ is generated as an algebra by $a, b, c, X, Y$, we conclude 
    \[H_\beta \simeq H(G, \chi, \chi^{-1}, e, f, \beta_3)/\langle X^n - \beta_1 (a^{nn_1} - b^n), Y^n - \beta_2 (a^{nn_1} - c^n) \rangle.\]
    The relations for $X^n$ and $Y^n$, written in terms of the generators $x, y, a, e, f$ are
    \[X^n - \beta_1 (a^{nn_1} - b^n) = (xa^{n_1})^n - \beta_1 (a^{nn_1} - b^n) = \pm a^{nn_1}(x^n \pm \beta_1 (1 - e^n))\]
    since $X^n = (xa^{n_1})^n = q^{n_1 n(n+1)/2}a^{nn_1}x^n = \pm a^{nn_1}x^n$, where the sign depends on the parity of $n_1 n(n + 1)/2$. Similarly, $Y^n - \beta_2 (a^{nn_1} - c^n) = \pm a^{nn_1}(y^n \pm \beta_2 (1 - f^n))$. More generally, we will see in Proposition~\ref{proposition:hopf algebra quotients} that for any group $G$ and Hopf algebra $H(G, \chi, \eta, b, c, \beta_3)$ with $q = \eta(b)$ a primitive $n$th root of unity, and for any scalars $\beta_1, \beta_2$, the ideal $I$ generated by $x^n - \beta_1 (1 - b^n)$ and $y^n - \beta_2 (1 - c^n)$ is a Hopf ideal. Such Hopf algebras directly generalize the class of Hopf algebras considered in \cite{Wang-Wu-Tan} and many properties will be shown in this wider context.
\end{example}
\begin{example}[Fantino \& Garc\'{i}a's pointed Hopf algebras over dihedral groups]
    In \cite[Theorem B]{Fantino-Garcia}, Fantino and Garc\'{i}a classified finite dimensional Hopf algebras over the dihedral group $\DD_m$ of order $2m$ with $m = 4t$ and $t \geq 3$ where 
    \[\DD_m = \langle (g, h) \mid g^2 = 1 = h^m,\ gh = h^{-1}g\rangle.\]
    One of these families of algebras, denoted by $A_{i,n}(\lambda)$, is generated by $g$ and $h$ and two other generators $x$ and $y$ such that 
    \begin{equation}\label{generators_Ain}
    xh = -hx,\quad yh = -hy,\quad yx = -xy,\quad yg = gx
    \end{equation}
    \[x^2 = \lambda(1 - h^{2i}),\quad y^2 = \lambda(1 - h^{-2i})\]
    and $x$ being $(1, h^i)$-primitive and $y$ being $(1, h^{-i})$-primitive, where $n = m/2 = 2t$ and $1 \leq i < n$ is odd. 

    Let $\chi$ be a character of $C_m = \langle h \rangle$, the cyclic group of order $m$, with $\chi(h) = -1$ and consider $H = H(C_m, \chi, \chi, h^i, h^{-i}, 0)$. Since $i$ is odd, $q = \chi(h^i) = (-1)^i = -1$ is the primitive root of unity of index 2. Then $H$ is generated over $\KK[C_m]$ by $u$ and $v$ subject to 
    \[uh = -hu,\quad vh = -hv,\quad vu = \chi(h^i)uv = -uv.\]
    The coalgebra structure is determined by $h$ being group-like, $u$ being $(1,h^i)$-primitive and $v$ being $(1, h^{-i})$-primitive. As we will see in Proposition~\ref{proposition:hopf algebra quotients},
    \[I = \langle u^2 - \lambda(1 - h^{2i}), v^2 - \lambda(1 - h^{-2i})\rangle\]
    is a Hopf ideal of $H$ and $\KK[C_m]$ is a subalgebra of $H(C_m, \chi, \chi, h^i, h^{-i}, 0)/I$. Denote by $\ou$ and $\ov$ the images of $u$ and $v$ in $H/I$ and let $C_2 = \langle g \rangle$ be the cyclic group of order $2$, which acts on $H/I$ by the automorphism $\hat{g}$ defined as 
    \[\hat{g}(h) = h^{-1}, \quad \hat{g}(\ou) = \ov, \quad \text{and} \quad \hat{g}(\ov) = \ou.\]
    The skew group ring $(H/I) \# C_2$ is then generated by $g, h, \ou, \ov$ which satisfy the relations (\ref{generators_Ain}) for $\ou$ instead of $x$ and $\ov$ instead of $y$, i.e.,
    \[A_{i,n}(\lambda) \simeq \left(H(C_m, \chi, \chi, h^i, h^{-i}, 0)/I\right) * C_2.\]
\end{example}
\section{Basic Properties}
In most of the examples discussed above, the group $G$ is finite or free abelian of finite rank. In such cases, the Hopf algebras possess nice ring-theoretical properties which we record here. The Gelfand-Kirillov dimension of a $\KK$-algebra $R$ is denoted by $\mathrm{GKdim}(R)$. In the rest of the paper, let $G$ be a group, $(\chi, b)$ and $(\eta, c)$ be two elements in $\widehat{G}\times Z(G)$ such that $\eta(b) = \chi(c)^{-1}$, and let $\beta \in \KK$ be a scalar.

The Hopf algebras $H(G, \chi, \eta, b, c, \beta)$ fall into two classes. The first class can be written as a skew group ring $\KK[x, y]\# G$ of a polynomial ring in two variables over $G$. 

\begin{proposition}\label{proposition:change of variables in the zero derivation case}
    Consider $H = H(G, \chi, \eta, b, c, \beta)$. If $\beta = 0$ or $c = b^{-1}$, then $H$ is generated over $\KK[G]$ by $x$ and $z = c^{-1}y$ such that $xz = zx$ and $z$ is $(c^{-1}, 1)$-primitive. Moreover, $H \simeq \KK[x, z]\# G$ as algebras, where $g \in G$ acts on $x$ via $\chi^{-1}$ and on $z$ via $\eta^{-1}$.
\end{proposition}

\begin{proof}
    For $z = c^{-1}y$, we calculate 
    \begin{equation}\label{equation:first remark}
        zx = c^{-1}\eta(b)xy + \beta c^{-1}(1 - bc) = \eta(b)\chi(c)xz + \beta (c^{-1} - b) = xz + \beta (c^{-1} - b).
    \end{equation}
    Hence, if $\beta = 0$ or $c = b^{-1}$, then $zx = xz$. Thus, the subalgebra generated by $x$ and $z$ is a commutative polynomial ring, and since $c$ is invertible, $H$ is generated by $G \cup \{x, z\}$ as an algebra. As 
    \[
        \Delta(z) = (c^{-1} \otimes c^{-1})\cdot (y \otimes 1 + c \otimes y) = z \otimes c^{-1} + 1 \otimes z
    \]
    holds, $z$ is $(c^{-1}, 1)$-primitive. Since every element of $H$ can be written as a sum of monomials of the form $x^i y^j g$ for some $i, j \geq 0$ and $g \in G$, we conclude that $H$ is a skew group ring over $\KK[x, z]$, that is $H(G, \chi, \eta, b, c, \beta) \simeq \KK[x, z]\# G$.
\end{proof}
In the following result, we prove a similar change of basis result in the nonzero derivation case. Recall that
 from Remark~\ref{remark:chi-eta-inverse}, we know that in case $\beta(1-bc)\neq 0$, then $\eta=\chi^{-1}$. In particular $\chi(c)=\eta(b)^{-1}=\chi(b)$.  
 
 For $q\in \KK$ and $n\geq 0$, the $q$-number $[n]_q$ is defined as 
\begin{equation}
[n]_q = \frac{q^n - 1}{q - 1} = \sum_{i=0}^{n-1} q^i.
\end{equation}

\begin{lemma}\label{lemma:change of variables in the nonzero derivation case}
    Suppose that $\beta \neq 0$ and $c \neq b^{-1}$, and consider $H = H(G, \chi, \chi^{-1}, b, c, \beta)$. 
    Then $H$ is generated over $\KK[G]$ by $x$ and $z = \beta^{-1}c^{-1}y$ such that $z$ is $(c^{-1}, 1)$-primitive and $zx = xz + (c^{-1} - b)$. Moreover, for any integer $n > 0$:
    \begin{align}
        zx^n  &= x^nz + \left( [n]_q c^{-1} - [n]_{q^{-1}} b \right)x^{n-1}, \label{equation:commutation relations 1} \\
        xz^n &= z^n x  + \left([n]_q b - [n]_{q^{-1}} c^{-1} \right)z^{n-1}, \label{equation:commutation relations 2}
    \end{align}
    where $q = \chi(c)^{-1}$.
\end{lemma}

\begin{proof}
    For $z = \beta^{-1}c^{-1}y$, we obtain from Equation~\eqref{equation:first remark} that $zx = xz + (c^{-1} - b)$. Since $y$ is $(1,c)$-primitive, it follows that $z$ is $(c^{-1}, 1)$-primitive. Equation~\eqref{equation:commutation relations 1} follows for $n=1$, since $[1]_q=1$.  Assume that \eqref{equation:commutation relations 1} holds for some $n \geq 1$. Then
    \begin{align*}
        zx^{n+1}    & =  \left( x^nz + \left( [n]_q c^{-1} - [n]_{q^{-1}} b \right)x^{n-1}\right) x\\
                    & = x^n (xz+c^{-1}-b) + \left( [n]_q c^{-1} - [n]_{q^{-1}} b \right)x^{n}\\
                    & = x^{n+1}z +\left(\left(\chi^{n}(c^{-1}) + [n]_q\right)c^{-1}  - \left(\chi^n(b) + [n]_{q^{-1}}\right)b\right) x^n\\
                    & = x^{n+1}z +\left( [n+1]_q c^{-1}  - [n+1]_{q^{-1}} b\right) x^n,
    \end{align*}
    where we use $q^{-1}=\chi(c) = \chi(b)$.
    One proves Equation~\eqref{equation:commutation relations 2} analogously.
\end{proof}

\begin{theorem}\label{theorem:properties of H}
    Consider $H = H(G, \chi, \eta, b, c, \beta)$. 
    \begin{enumerate}
        \item A $\KK$-basis for $H$ is $\{gx^i y^j \mid g \in G,\ i, j \geq 0\}$.
        \item The Hopf algebra $H$ falls into one of two isomorphism classes:
        \begin{itemize}
            \item If $\beta(1-bc) = 0$ (i.e., $\beta = 0$ or $c = b^{-1}$), then 
            \[H \simeq H(G, \chi, \eta, b, c, 0).\]
            \item If $\beta(1-bc) \neq 0$ (i.e., $\beta \neq 0$ and $c \neq b^{-1}$), then necessarily $\eta = \chi^{-1}$ and 
            \[H \simeq H(G, \chi, \chi^{-1}, b, c, 1).\]
        \end{itemize}
        \item $H$ is Noetherian provided $G$ is polycyclic-by-finite.
        \item $H$ is a domain if and only if $\KK[G]$ is a domain.
        \item If $G$ is a finitely generated group with $\mathrm{GKdim}(\KK[G]) = d < \infty$, then $H$ is an affine algebra of Gelfand-Kirillov dimension $d + 2$.
        \item The antipode of $H$ satisfies $S^{2}|_G = \mathrm{id}$ and 
        \[S^{2}(x) = \chi(b)x, \quad S^{2}(y) = \eta(c)y.\]
        In particular, $S$ has finite order if and only if $\chi(b)$ and $\eta(c)$ are roots of unity.
        \item $H$ is a pointed Hopf algebra with group-likes $G(H) = G$.
    \end{enumerate}
\end{theorem}

\begin{proof}
    (1) This is clear since $H$ is an iterated Ore extension.

    (2) If $\beta = 0$ or $c = b^{-1}$, then $\beta(1-bc) = 0$. By Proposition~\ref{proposition:change of variables in the zero derivation case}, the substitution $z = c^{-1}y$ (or $z=y$ if $\beta \neq 0$) shows $H \simeq H(G, \chi, \eta, b, c, 0)$. 
    
    Now assume $\beta \neq 0$ and $c \neq b^{-1}$. As we remarked earlier, in this case $\eta = \chi^{-1}$. From Lemma~\ref{lemma:change of variables in the nonzero derivation case}, $H \simeq H(G, \chi, \chi^{-1}, b, c, 1)$.

    (3) If $G$ is polycyclic-by-finite, then $\KK[G]$ is Noetherian (\cite[1.5.12]{McConnell-Robson}). It follows that $H$ is Noetherian by Hilbert's Basis Theorem.

    (4) This is clear.

    (5) If $G$ is a finitely generated group, then $\KK[G]$ and hence $H$ are affine $\KK$-algebras. By \cite[Theorem 2.6]{Brown-O'Hagan-Zhang}, $\mathrm{GKdim}(H) = d + 2$.

    (6) As $x$ is $(1, b)$-primitive, $S(x) = -b^{-1}x$. Hence $S^{2}(x) = b^{-1}xb = \chi(b)x$, so there exists $n$ such that $S^n(x) = x$ if and only if $n$ is even and $\chi(b)$ is a root of unity. Similarly, $S^n(y) = y$ if and only if $n$ is even and $\eta(c)$ is a root of unity. Conversely, if $\chi(b)$ and $\eta(c)$ are roots of unity of orders $k$ and $l$ respectively, then $S$ has order $2\mathrm{lcm}(k,l)$.

    (7) This follows from \cite[Proposition 2.5]{Brown-O'Hagan-Zhang}.
\end{proof}
\begin{corollary}
    Consider $H = H(G, \chi, \eta, b, c, \beta)$ such that $G$ is finitely generated abelian and $\chi$ and $\eta$ have finite order. If  $\beta(1-bc)=0$  or $q = \eta(b) \neq 1$, then $H$ is finitely generated over a central affine subalgebra. In particular, $H$ is an affine Noetherian PI-algebra, and all simple left or right $H$-modules are finite dimensional.
\end{corollary}

\begin{proof}
    Let $g_1, \ldots, g_d$ be a set of generators of $G$. Since $\chi$ and $\eta$ have finite order, we have $\chi(g)^{n} = 1 = \eta(g)^n$ for all $g \in G$, where $n$ is the least common multiple of the orders of $\chi$ and $\eta$. Thus, $g_{i}^{n}$ commutes with $x$ and $y$ for all $i = 1, \ldots, d$. 

    If $\beta(1-bc)=0$, then by Proposition~\ref{proposition:change of variables in the zero derivation case}, $H \simeq H(G, \chi, \eta, b, c, 0) \simeq \KK[x, z]\# G$ for $z = c^{-1}y$. Thus, $x$ and $z$ commute, so $x^n$ and $z^n$ are central in $H$ and $H$ contains the central subalgebra $C = \KK[g_{1}^{\pm n}, \ldots, g_{d}^{\pm n}, x^n, z^n]$. 

    In the case $\beta \neq 0$ and $c^{-1} \neq b$, by Theorem~\ref{theorem:properties of H}, $H \simeq H(G, \chi, \chi^{-1}, b, c, 1)$. 
    Since $q\neq 1$ and $q^n=1$, we obtain $[n]_q=1=[n]_{q^{-1}}$. From Equations~\eqref{equation:commutation relations 1} and \eqref{equation:commutation relations 2}, we conclude $zx^n = x^nz$ and $xz^n = z^nx$. Thus, $x^n$ and $z^n$ are central and  $C = \KK[g_{1}^{\pm n}, \ldots, g_{d}^{\pm n}, x^n, z^n]$ is a central subalgebra of $H$.

    In each case, $H$ is finitely generated over the affine central subalgebra $C$. Therefore, $H$ is an affine PI-algebra, and any simple $H$-module is finite dimensional over $\KK$ by \cite[13.10.3]{McConnell-Robson}.
\end{proof}
We have seen that some of the examples in Section~\ref{section:examples}, like Example~\ref{example_sl2}, were quotients of Hopf algebras of the type $H(G, \chi, \eta, b, c, \beta)$. We collect basic properties of these quotients here.

\begin{proposition}\label{proposition:hopf algebra quotients}
    Consider $H = H(G, \chi, \eta, b, c, \beta)$ and let $p = \chi(b)$ and $r = \eta(c)$. Suppose that $p$ is a primitive $n$th root of unity and $r$ is a primitive $m$-th root of unity for some $n, m > 1$. Then for any $\lambda_1, \lambda_2 \in \KK$, the ideal 
    \[I = \langle x^n - \lambda_1 (1 - b^n), y^m - \lambda_2 (1 - c^m) \rangle\]
    is a Hopf ideal of $H$, and the quotient $H/I$ is free of finite rank over the Hopf subalgebra $\KK[G]$. We denote this quotient by $H(G, \chi, \eta, b, c, \beta, \lambda_1, \lambda_2)$.
\end{proposition}

\begin{proof}
    We need to show that $I$ is a Hopf ideal, that is, $\Delta(I) \subseteq I \otimes H + H \otimes I$, $\epsilon(I) = 0$, and $S(I) \subseteq I$. Recall that if $u$ and $v$ are elements in a $\KK$-algebra and $q \in \KK$ is a scalar such that $vu = quv$, then for all $n > 0$,
    \[(u + v)^n = \sum_{k=0}^{n}\binom{n}{k}_q u^k v^{n-k},\]
    where $\binom{n}{k}_q = \frac{(n!)_q}{(k!)_q (n - k)!_q}$ are the Gauss polynomials and $(k!)_q = [1]_q [2]_q \cdots [k]_q$ (see \cite[Proposition IV.2.2]{Kassel}). Note that $\binom{n}{0}_q = \binom{n}{n}_q = 1$. Since $p$ is a primitive $n$th root of unity, $(n!)_p = 0$ while $(k!)_p \neq 0$ for $0 < k < n$, implying $\binom{n}{k}_p = 0$ for all $0 < k < n$. Thus, if $vu = puv$, then $(u + v)^n = u^n + v^n$. 
    
    This applies in particular to $u = x \otimes 1$ and $v = b \otimes x$. Since $\chi(b)=p$, we have $p^{-n} = 1$, and calculating the commutation relation gives
    \[
        vu = (b \otimes x)(x \otimes 1) = bx \otimes x = p^{-1} xb \otimes x = p^{-1} (x \otimes 1)(b \otimes x) = p^{-1}uv.
    \]
    Since $p$ is a primitive $n$th root of unity, so is $p^{-1}$. Thus we have
    \begin{align*}
        \Delta(x^n - \lambda_1 (1 - b^n)) & = (x \otimes 1 + b \otimes x)^n - \lambda_1 (1 \otimes 1 - b^n \otimes b^n)\\
        & = x^n \otimes 1 + b^n \otimes x^n - \lambda_1 (1 - b^n) \otimes 1 - \lambda_1 b^n \otimes (1 - b^n)\\
        & = (x^n - \lambda_1 (1 - b^n)) \otimes 1 + b^n \otimes (x^n - \lambda_1 (1 - b^n)).
    \end{align*}
    Thus, $\Delta(x^n - \lambda_1 (1 - b^n)) \in I \otimes H + H \otimes I$. Similarly, $\Delta(y^m - \lambda_2 (1 - c^m)) \in I \otimes H + H \otimes I$ holds, from which $\Delta(I) \subseteq I\otimes H + H \otimes I$ follows. Clearly, the counit satisfies $\epsilon(x^n - \lambda_1 (1 - b^n)) = \epsilon(y^m - \lambda_2 (1 - c^m)) = 0$. 
    
    Recall that if $vu = quv$, then $(vu)^k = q^{k(k+1)/2}u^k v^k$ for any $k \geq 1$. Since $xb = pbx$, we have $S(x) = -b^{-1}x = -p x b^{-1}$, and
    \[ S(x)^n = (-b^{-1}x)^n = (-1)^n p^{n(n+1)/2}x^n b^{-n}. \]
    If $n$ is odd, then $p^{n(n+1)/2} = (p^n)^{(n+1)/2} = 1$. If $n$ is even, then $p^{n(n+1)/2} = (p^{n/2})^{n+1} = (-1)^{n+1} = -1$ (since $p$ is a primitive $n$th root of unity). In either case, $(-1)^n p^{n(n+1)/2} = -1$. Therefore, noting that $S(b^n) = b^{-n}$, we get
    \[
        S(x^n - \lambda_1 (1 - b^n)) = -x^n b^{-n} - \lambda_1 (1 - b^{-n}) = -b^{-n}(x^n - \lambda_1 (1 - b^n)) \in I.
    \]
    The proof that $S(y^m - \lambda_2 (1 - c^m)) \in I$ is analogous. Hence $I$ is a Hopf ideal of $H$. Since $\{gx^i y^j \mid g \in G,\ i,j \geq 0\}$ is a $\KK$-basis of $H$, we have $\KK[G] \cap I = 0$, so $\KK[G]$ injects into $H/I$. Denoting the images of $x$ and $y$ in the quotient by the same symbols, we see $x^n, y^m \in \KK[G]$. Therefore, $H/I$ is generated by $\{x^i y^j \mid 0 \leq i < n, 0 \leq j < m\}$ as a module over $\KK[G]$.
\end{proof}

The Hopf algebras $H(G, \chi, \eta, b, c, \beta, \lambda_1, \lambda_2)$ are finite modules over the group algebra $\KK[G]$ and hence finite dimensional in case $G$ is a finite group. However, infinite dimensional Hopf algebras of that type are also of interest as the Hopf algebras $H_\beta$ of Example~\ref{example_sl2}, studied in \cite{Wang-Wu-Tan} show. Recall that 
\[H_\beta = H(\ZZ^3, \chi, \chi^{-1}, e, f, -\beta_3)/\langle x^n + \beta_2 (1 - e^n), y^n + \beta_3 (1 - f^n)\rangle.\]
Hence, $H_\beta = H(\ZZ^3, \chi, \chi^{-1}, e, f, -\beta_3, -\beta_1, -\beta_2)$ after some change of variables. Some of the structural results of $H_\beta$ obtained in \cite{Wang-Wu-Tan} follow, therefore, from our results. 

Recall that an algebra is \emph{Gorenstein} if it has finite left and right injective dimension over itself (see for example \cite{Wu-Zhang}). Moreover, a Noetherian algebra $A$ is left \emph{Artin-Schelter-Gorenstein} if $_{A}A$ has finite injective dimension, say $d$, and for every simple $A$-module $V$, $\mathrm{Ext}_{A}^{i}(V, A) = 0$ for all $i \neq d$ and $\mathrm{Ext}_{A}^{d}(V, A)$ is a simple $A^{\circ}$-module. By \cite[Theorem 3.5]{Wu-Zhang}, a Noetherian PI Hopf algebra $A$ over $\KK$ such that every simple $A$-module is finite dimensional over $\KK$ is Artin-Schelter-Gorenstein.

\begin{proposition}
    Let $d \geq 1$ and $H = H(\ZZ^d, \chi, \eta, b, c, \beta, \lambda_1, \lambda_2)$ for some characters $\chi, \eta$ of $\ZZ^d$ such that $\chi(c) \neq 1 \neq \eta(b)$ and $\chi(c)^n = 1 = \eta(b)^m$ for some $m, n$. 
    \begin{enumerate}
        \item Any simple left or right $H$-module is finite dimensional.
        \item $H$ is a pointed affine Noetherian PI-Hopf algebra of Gelfand-Kirillov dimension $d$, Artin-Schelter-Gorenstein of injective dimension $d$ as well as finitely generated as a module over $\KK[G]$. 
    \end{enumerate}
\end{proposition}
\begin{proof}
    (1) Recall that $H(\ZZ^d, \chi, \eta, b, c, \beta, \lambda_1, \lambda_2) = H(\ZZ^d, \chi,\eta, b, c, \beta)/I$ where $I = \langle x^n - \lambda_1 (1 - b^n), y^m - \lambda_2 (1 - c^m)\rangle$. Then $H$ is finitely generated by $\{x^i y^j \mid 0 \leq i < n,\ 0 \leq j < m\}$ as a module over $\KK[G]$. In particular, since $\KK[G]$ is an affine commutative domain, $H$ is an affine PI-algebra and hence any simple $H$-module is finite dimensional over $\KK$ by \cite[13.10.3]{McConnell-Robson}.

    (2) By \cite[Theorem 3.5]{Wu-Zhang}, a Noetherian PI-Hopf algebra $A$ over $\KK$ such that every simple $A$-module is finite dimensional over $\KK$ is Artin-Schelter-Gorenstein. 
\end{proof}
\section{Finite Dimensional Simple $H$-modules}\label{section:finite dimensional simple $H$-modules}

In this section, we classify the finite dimensional simple modules over the algebra $H = H(G, \chi, \eta, b, c, \beta)$. Denote by $\Irr(H)$ the set of isomorphism classes of finite dimensional simple left $H$-modules. 
For this section we make the following {\bf general assumptions}:
\begin{itemize}
\item $\KK$ is an algebraically closed field of characteristic zero.
\item $G$ is a finitely generated abelian group.
\item The characters $\chi$ and $\eta$ have finite order.  
\end{itemize}

We have seen that in Proposition \ref{proposition:change of variables in the zero derivation case} that in case $\beta=0$ or $b=c^{-1}$, we can replace $y$ with $c^{-1}y$ and obtain that $x$ and $y$ commute 
or that we can replace otherwise $y$ by $\beta^{-1}c^{-1}y$, to obtain $[y,x]=c^{-1}-b$ (see Lemma \ref{lemma:change of variables in the nonzero derivation case}). Hence we assume that $H$ is generated as an algebra by $x$, $y$, and the elements of $G$, subject to the relations:
\[
    xg = \chi(g)gx, \qquad yg = \eta(g)gy, \qquad yx = xy + e,
\]
for all $g \in G$, where  $e=c^{-1}-b$, in case $\beta(c^{-1}-b)\neq 0$ and $e=0$, otherwise.
Since the descriptions of the simple modules differ significantly depending on $e$, we divide the classification into two parts:
\begin{enumerate}
    \item $\beta e =0$ : the Skew Group Ring Case;
    \item $\beta e \neq 0$: the Differential Operator Ring Case.
\end{enumerate}

Before proceeding, we recall some necessary terminology. Consider the sets $S_x = \{x^i \mid i \ge 1\}$ and $S_y = \{y^i \mid i \ge 1\}$. It is easy to see that $S_x$ is a left denominator set in $\KK[G][x; \tau_\chi]$. That $S_x$ is also a left denominator set in $H$ follows from \cite[Lemma 1.4]{Goodearl}. By the symmetry between the roles of $x$ and $y$ in the construction of $H$, $S_y$ too is a left denominator set in $H$. For a left denominator set $S \subset H$ and a left $H$-module $M$, the $S$-\emph{torsion submodule} of $M$ with respect to $S$ is defined as
\[ \mathrm{ass}_S M = \{ m \in M \mid sm=0 \text{ for some } s \in S\}. \]
The module $M$ is called $S$-\emph{torsion} if $\mathrm{ass}_S M = M$, and $S$-\emph{torsion-free} if $\mathrm{ass}_S M = 0$. Specifically:
\begin{itemize}
    \item[(i)] We call $M$ \emph{$x$-torsion} (resp. \emph{$x$-torsion-free}) if it is torsion (resp. torsion-free) with respect to $S_x$.
    \item[(ii)] We call $M$ \emph{$y$-torsion} (resp. \emph{$y$-torsion-free}) if it is torsion (resp. torsion-free) with respect to $S_y$.
    \item[(iii)] We say $M$ is \emph{torsion} if it is both $x$-torsion and $y$-torsion.
\end{itemize}

\subsection{The Skew Group Ring Case}
In this subsection, we assume that either $\beta = 0$ or $e = 0$, which implies that $yx = xy$. We then consider the algebra $H = \KK[x, y] \# G$. Our goal is to classify the finite-dimensional simple $H$-modules based on their torsion properties as defined earlier.

We will work in a more general setting and consider a skew group ring of an affine commutative $\KK$-algebra $R$ generated by a set $X = \{x_1, \ldots, x_n \}$ and a finitely generated abelian group $G$. Recall that a group $G$ acts on a ring $R$ by automorphisms if there exists a group homomorphism $\rho:G \to \text{Aut}(R)$. We say that $G$ \emph{acts faithfully} if $\ker(\rho) = \{e\}$ is trivial, where $e$ is the neutral element of $G$. Any group action $\rho:G \to \text{Aut}(R)$ induces a faithful group action $\rho:G/N \to \text{Aut}(R)$ with $N = \ker(\rho)$. Let $G$ act on $R$ via a set of characters $\chi = \{\chi_1, \ldots, \chi_n\}$, that is, $\chi_i \in \widehat{G}$ and
\[g\cdot x_i = \chi_{i}^{-1}(g)x_i\]
for all $g \in G$ and $1 \leq i \leq n$. Let $R\# G$ be the corresponding skew group ring. That is, $R\# G$ is the $\KK$-algebra generated by $G$ and $X$ subject to the relations
\[gx_i = \chi_{i}(g)^{-1}x_i g\]
or, equivalently $x_i g = \chi_{i}(g)gx_i$, for all $g \in G$ and $1 \leq i \leq n$. 

For a character $\lambda \in \widehat{G}$, let $\KK_\lambda$ be the one dimensional $\KK[G]$-module where the action of $G$ is given by $g\cdot v = \lambda(g)v$ for all $g \in G$ and $v \in \KK_\lambda$.
\begin{lemma}\label{lemma:torsion modules in zero derivation case}
    Every simple torsion $R\# G$-module $V$ is a simple $\KK[G]$-module. In particular, $\widehat{G}\subseteq \Irr(H)$ and since $G$ is finitely generated abelian, $V$ is one dimensional and isomorphic to $\KK_\lambda$ for some $\lambda \in \widehat{G}$.
\end{lemma}
\begin{proof}
    Let $V$ be a simple $R\# G$-module. Since each $x_i$ is normal in $R\#G$, for each $i$ we have either $x_iV=0$ or $x_iV=V$. If $V$ is torsion with respect to $S_{x_i}$ for every $i$, then necessarily $x_iV=0$ for all $i$, hence $\langle X\rangle V=0$. Therefore $V$ is an $R\#G/\langle X\rangle \simeq \KK[G]$-module. If $G$ is finitely generated abelian, $\KK[G]$ is affine commutative and hence $V$ is one dimensional. 
\end{proof}
In many examples in the literature, the group $G$ acts on an algebra by multiplication by roots of unity. The next result shows that this is not unusual.
\begin{lemma}
    If there exists a finite dimensional $x_i$-torsion-free simple left $R\# G$-module, then $\chi_i$ has finite order.
\end{lemma}
\begin{proof}
    Let $V$ be a finite dimensional simple left $R\# G$-module that is $x_i$-torsion-free and consider the left multiplication by $x_i$ on $V$ as a linear map $l_{x_i} :V \to V$. Since $V$ is $x_i$-torsion-free, $l_{x_i}$ has an eigenvector $v$ with a nonzero eigenvalue $\lambda$.  For any $g\in G$ and $t\geq 0$, the element $g^tv$ is an eigenvector of $l_{x_i}$ with eigenvalue $\lambda \chi_{i}^{t}(g)$, since  $x_i \cdot g^t v = \lambda \chi_{i}^{t}(g)g^t v$. As eigenvectors of different eigenvalues are linearly independent and $V$ has finite dimnension, say $n$, there exists $0\leq t \leq n$ such that $\chi_{i}^{t}(g)=1$. Hence $\chi_{i}$ has finite order $\leq n$.
\end{proof}
\begin{remark}
    In light of Lemma~\ref{lemma:torsion modules in zero derivation case}, when classifying simple modules in the zero derivation case, it suffices to consider modules that are $X$-torsion-free, since we can always reduce to this case. Here, by $X$-torsion-free, we mean a module that is $x_i$-torsion-free for every $x_i \in X$. Hence, we will also assume that the characters $\chi_1, \ldots, \chi_n$ have finite order.
\end{remark}
Let $N = \cap_{i=1}^{n}\ker(\chi_i)$. Since the $\chi_i$ have finite order and $G$ is finitely generated abelian, the quotient group $G/N$ is finite. Following \cite[Lemma 1.3]{Passman}, we view $R\# G$ as a crossed product
\[R\# G \simeq R[N]*(G/N)\]
where the action of $G/N$ on $R[N]$ is induced from the action of $G/N$ on $R$. Fix a set $\Lambda \subseteq G$ of coset representatives of $G/N$. Then $R\# G$ is a free left $R[N]$-module with basis $\{\overline{g} \mid g \in \Lambda\}$ and such that 
\[\overline{g}\overline{h} = \gamma(g, h)\overline{gh} \qquad g,h\in\Lambda\]
for a 2-cocycle $\gamma: G/N \times G/N \to U(R[N])$, the group of units of $R[N]$.
\begin{proposition}\label{proposition:torsionfree simples over general skew group ring}
    Let $G, R, X, \chi_i$, and $N$ be as above. Let $V = \KK v$ be an $X$-torsion-free simple $R[N]$-module. Then:
    \begin{enumerate}
        \item The induced module $\overline{V} = (R\# G) \otimes_{R[N]}V$ is a simple $R\# G$-module. 
        \item A $\KK$-basis for $\overline{V}$ is $\{v_g := \overline{g}\otimes v \mid  g \in \Lambda\}$. Thus $\dim(\overline{V}) = |G/N|$. 
        \item The action of $R\# G$ on $\overline{V}$ is given by:
        \[x_i \cdot v_g = \chi_{i}(g)k_i v_g, \quad \forall 1 \leq i \leq n\]
        \[hg\cdot v_{g'} = \lambda(\gamma(g, g')h)v_{gg'}, \quad \forall g, g' \in \Lambda, h \in N\]
        where $x_i \cdot v = k_i v$ ($k_i \in \KK^\times$) and $h\cdot v = \lambda(h)v$, for some $\lambda \in \widehat{N}$. 
        \item Any $X$-torsion-free simple left $R\# G$-module is isomorphic to $\overline{V}$ for some $X$-torsion-free simple $R[N]$-module $V$.
    \end{enumerate}
    We denote this module by $V(\mathbf{k}, \lambda)$, where $\mathbf{k} = (k_1, \ldots, k_n)$.
\end{proposition}
%%%%%
\begin{proof}
    (1 and 2) As we noted before the proposition, $R[N]$ is affine commutative and since $\KK$ is algebraically closed of characteristic zero, $V = \KK v$ is one-dimensional for some nonzero $v \in V$. Consider the induced $R\# G$-module 
    \[\overline{V} = (R\# G) \otimes_{R[N]}V = \oplus_{g \in \Lambda}\KK\overline{g} \otimes v\]
    Let $W$ be a nonzero submodule of $\overline{V}$ and let $w = \sum r_g \overline{g}\otimes v$ be a nonzero element of $W$ of minimal support $\Lambda' \subseteq \Lambda$, where by \emph{support} we mean the set of those $g$ for which $r_g \neq 0$. We claim that $\Lambda'$ is a singleton. Since $V$ is one-dimensional, for any $1 \leq i \leq n$, there exists $k_i \in \KK$ such that $x_i \cdot v = k_i v$. Since $V$ is $x_i$-torsion-free, $k_i \neq 0$ and 
    \[k_{i}^{-1}x_i \cdot w = \sum_{g \in \Lambda'}r_g \chi_{i}(g)\overline{g}\otimes k_{i}^{-1}x_i v = \sum_{g\in \Lambda'}r_g \chi_{i}(g)\overline{g} \otimes v\]
    Hence, for any $h \in \Lambda'$, we get
    \[(\chi_{i}(h) - k_{i}^{-1}x_{i})w = \sum_{g\neq h}r_g (\chi_{i}(h) - \chi_{i}(g))\overline{g} \otimes v.\]
    By the minimality of the support of $w$, we must have $\chi_{i}(g) = \chi_{i}(h)$ for all $g, h \in \Lambda'$ and $1 \leq i \leq n$. Thus $gh^{-1}\in N$ for any $g, h \in \Lambda'$ which means that any two elements of $\Lambda'$ represent the same coset, that is $\Lambda'$ is a singleton and $w = r_g \overline{g} \otimes v$ for some $g \in \Lambda$ and $r_g \neq 0$. Multiplying by $r_{g}^{-1}\overline{g^{-1}}$, we can assume that $W$ contains the element $\overline{e} \otimes v$, from which we can generate all elements of the form $\overline{g} \otimes v$ with $g \in \Lambda$.

    (3) By hypothesis, $V$ is a one-dimensional $R[N]$-module. Hence, elements of $N$ act via a character $\lambda \in \widehat{N}$ and the generators $x_i$ act by scalars $k_i$ on $V$. As above, $x_i v_g = \chi_{i}(g)k_i v_g$ for any $g \in \Lambda$ and $1 \leq i \leq n$. Furthermore, any element of $G$ can be written as $hg$ with $h \in N$ and $g \in \Lambda$. Thus, for any $g' \in \Lambda$, we have
    \[\overline{hg}\cdot v_{g'} = \overline{h}\overline{g}\overline{g'}\otimes v = \overline{gg'} \otimes \gamma(g, g')hv = \lambda(\gamma(g,g')h)v_{gg'}.\]

    (4) Given a finite dimensional $X$-torsion-free simple left $R\# G$-module $M$, we can choose a simple left $R[N]$-module $V = \KK v \subseteq M$. Then $f: \overline{V} \to M$ given by $\overline{g}\otimes v \mapsto \overline{g}v$ for $g \in \Lambda$ is nonzero and left $R\# G$-linear. Since $M$ and $\overline{V}$ are simple, it follows that $f$ is an isomorphism. 
\end{proof}
We apply this result to $H = H(G, \chi, \eta, b, c, 0)$.
\begin{corollary}\label{corollary:x-torsionfree y-torsion simples in zero derivation}
    Consider $H = H(G, \chi, \eta, b, c, 0)$ with $G$ finitely generated abelian. Suppose $q = \chi(c)=\eta(b)^{-1}$ is a primitive $n$th root of unity and $\mathrm{Im}(\chi) = \langle q \rangle$. Let $V$ be an $x$-torsion-free, $y$-torsion simple $H$-module. Then there exist $(\alpha, \lambda) \in \KK^\times \times \widehat{N}$ with $N = \ker(\chi)$, such that $V$ has a $\KK$-basis $\{v_0, \ldots, v_{n-1}\}$ and 
    \[y\cdot v_i = 0,\quad x\cdot v_i = q^i \alpha v_i,\quad h\cdot v_i = \lambda(h)v_i,\]
    and
    \[ c\cdot v_i = \begin{cases}
        \lambda(c^n)v_0 & \text{ if } i = n-1\\ v_{i+1} & \text{ otherwise}
    \end{cases}\]
    for all $0 \leq i \leq n-1$ and $h \in N$. We denote $V$ by $V_{x}(\alpha, \lambda)$. Moreover, $V_{x}(\alpha, \lambda) \simeq V_{x}(\alpha', \lambda')$ as $H$-modules if and only if $\lambda = \lambda'$ and $\alpha = q^i \alpha'$ for some $i$.
\end{corollary}
\begin{proof}
    By Proposition~\ref{proposition:change of variables in the zero derivation case}, $H$ can be written as $\KK[x,y]\# G$ with $x$ being $(1, b)$-primitive and $y$ being $(c^{-1}, 1)$-primitive. Let $V$ be an $x$-torsion-free, $y$-torsion simple $H$-module. Note that $V$ is finite dimensional as $H$ is an affine PI algebra. Moreover, since $yV = 0$, $V$ is an $x$-torsion-free simple $\KK[x]\# G$-module where $G$ acts on $x$ via $\chi$. By hypothesis, $G/N \simeq \langle q\rangle$ is finite cyclic of order $n$, where $N = \ker(\chi)$. Furthermore, since $\mathrm{Im}(\chi)=\langle q\rangle$ and $\chi(c)=q$, the coset $cN$ generates $G/N$, hence $\{1,c,\ldots,c^{n-1}\}$ is a set of representatives. By Proposition~\ref{proposition:torsionfree simples over general skew group ring}, $V \simeq V_x (\alpha, \lambda)$ for some $\alpha \in \KK^\times$ and $\lambda \in \widehat{N}$ such that $V$ has a basis $\{v_i := v_{c^i} \mid 0 \leq i \leq n-1\}$. By the same proposition, $x\cdot v_i = q^i \alpha v_i$ and $h\cdot v_i = \lambda(h)v_i$ for all $h \in N$. Moreover, if $i+j < n$, then $c^i v_j = v_{i+j}$. For $i + j \geq n$, we have $c^i c^j = c^n c^{i+j-n}$ and $\gamma(c^i, c^j) = c^n$. Thus, if $i + j \geq n$, then $c^i v_j = \lambda(c^n)v_{i+j}$ by Proposition~\ref{proposition:torsionfree simples over general skew group ring}, where the subscript $i + j$ is calculated modulo $n$. 

    Consider the modules $V = V_x (\alpha, \lambda)$ and $V' = V_x (\alpha', \lambda')$ for some $\alpha, \alpha' \in \KK^\times$ and $\lambda, \lambda' \in \widehat{N}$. Suppose that $f:V \to V'$ is an isomorphism of left $H$-modules. Then $f(v_0) = \sum_{i=0}^{n-1}\mu_i v'_{i}$ where $\{v_{0}', \ldots, v_{n-1}'\}$ is the basis for $V'$. Hence
    \[\sum_{i=0}^{n-1}\alpha\mu_i v_{i}' = f(\alpha v_0) = x\cdot f(v_0) = \sum_{i=0}^{n-1}\mu_i \alpha' q^i v_{i}'\]
    and $\alpha = \alpha' q^i$ for all $\mu_i \neq 0$. Since $f(v_0)$ is nonzero and since $q$ is a primitive $n$th root of unity, there exists precisely one index $i$ such that $\alpha = \alpha' q^i$ and $f(v_0) = \mu_i v_{i}'$. Moreover, for any $h \in N$, $\lambda(h)\mu_i v_{i}' = f(h\cdot v_0) = \mu_i \lambda'(h)v_{i}'$. Thus $\lambda = \lambda'$. Conversely, if $\lambda = \lambda'$ and $\alpha = \alpha' q^i$ for some $i$, then $f:V \to V'$ defined by $v_0 \mapsto v_{i}'$ is an isomorphism of $H$-modules.
\end{proof}
With an analogous proof, we can determine the $y$-torsion-free, $x$-torsion simple $H$-modules. Note, however, that when writing $H(G, \chi, \eta, b, c, 0) \simeq \KK[x,y]\# G$ as a skew group ring, $y$ is $(c^{-1}, 1)$-primitive.
\begin{corollary}\label{cor:y-torsionfree-x-torsion-simples}
Consider $H=H(G,\chi,\eta,b,c,0)$ with $G$ finitely generated abelian.
Suppose $q=\eta(b)$ is a primitive $n$th root of unity and
$\mathrm{Im}(\eta)=\langle q\rangle$.
Let $V$ be a $y$-torsion-free, $x$-torsion simple $H$-module.
Then there exist $(\alpha,\lambda)\in \KK^\times\times\widehat{N}_y$ with
$N_y=\ker(\eta)$ such that $V$ has a $\KK$-basis $\{v_0,\ldots,v_{n-1}\}$ and
\[
x\cdot v_i=0,\qquad y\cdot v_i=q^i\alpha\,v_i,\qquad h\cdot v_i=\lambda(h)v_i,
\]
and
\[
b\cdot v_i=\begin{cases}
\lambda(b^n)v_0 & \text{if } i=n-1,\\
v_{i+1} & \text{ otherwise}
\end{cases}
\]
for all $0\le i\le n-1$ and $h\in N_y$.
We denote $V$ by $V_y(\alpha,\lambda)$.
Moreover, $V_y(\alpha,\lambda)\simeq V_y(\alpha',\lambda')$ if and only if
$\lambda=\lambda'$ and $\alpha=q^i\alpha'$ for some $i$.
\end{corollary}

%%%%%%%%%%%%%%%%%%%%%%%%%%%%%%%%%%%%%%%%%%%%%%%%%%%%%%%%%%%%
The structure of $x$- and $y$-torsion-free simple modules seems to be more complicated. From Proposition~\ref{proposition:torsionfree simples over general skew group ring}, we conclude that the $x$- and $y$-torsion-free simple $H$-modules have a basis that is indexed by a set of representatives of cosets of $G/N$ where 
\[N = \ker(\chi) \cap \ker(\eta).\]

We restrict ourselves to the case $\eta = \chi^t$ with $t$ relatively prime to the order $n$ of $\chi$. In this case, 
$N = \ker(\chi)  = \ker(\eta)$ and $\mathrm{Im}(\chi)=\mathrm{Im}(\eta) = \langle q\rangle$. 
For any  $g \in G$ there exists $i \in \ZZ$ such that $\chi(g) = \chi(c)^i = q^{-i}$ and $gc^{-i} \in N$. Thus $G/N = \{c^i N \mid i \in \ZZ_n\}$ has order $n$.  From Proposition~\ref{proposition:torsionfree simples over general skew group ring} we obtain:
\begin{corollary}\label{x-y-torsionfree-skewgroupcase}
Let $H=H(G,\chi,\chi^t,b,c,0)$, where $G$ is a finitely generated abelian group,
$q=\eta(b) = \chi(c)^{-1}$ is a primitive $n$th root of unity,
$\langle q\rangle=\mathrm{Im}(\chi)$, and $t$ is relatively prime to $n$.
Let $N=\ker(\chi)$.
Then any $x$- and $y$-torsion-free simple $H$-module $V$ has dimension $n$
and is isomorphic to an induced module $H\otimes_{R[N]}M$ for some
one-dimensional $R[N]$-module, where $R=\KK[x,y]$.
In particular, there exist $\alpha_x,\alpha_y\in\KK^\times$ and
$\lambda\in\widehat{N}$ such that $V$ has a $\KK$-basis
$\{v_0,\ldots,v_{n-1}\}$ and
\[
x\cdot v_i=\alpha_x q^{-i}v_i,\qquad
y\cdot v_i=\alpha_y q^{-it}v_i,\qquad
c^j h\cdot v_i=\lambda(h)v_{i+j},
\]
for all $0\le i,j\le n-1$ and $h\in N$, where indices are taken modulo $n$.
We denote $V$ by $V_{xy}(\alpha_x,\alpha_y,\lambda)$.
Moreover,
\[
V_{xy}(\alpha_x,\alpha_y,\lambda)\simeq
V_{xy}(\alpha_x',\alpha_y',\lambda')
\]
if and only if $\lambda=\lambda'$, $\alpha_x=q^{-i}\alpha_x'$, and
$\alpha_y=q^{-it}\alpha_y'$ for some $i$.
\end{corollary}

From Lemma \ref{lemma:torsion modules in zero derivation case} and Corollaries \ref{corollary:x-torsionfree y-torsion simples in zero derivation}, \ref{cor:y-torsionfree-x-torsion-simples} and \ref{x-y-torsionfree-skewgroupcase} we conclude  for $H = H(G, \chi, \chi^t, b, c, 0)$ with $G$ finitely generated abelian, $q = \chi(c)$ a primitive $n$th-root of unity, $t$ relatively prime to $n$ and $\mathrm{Im}(\chi)=\langle q \rangle$, that the finite dimensional simple $H$-modules are stratified as follows:
\begin{equation}
\Irr(H) =  \widehat{G} \sqcup \left( U \times \widehat{N}\right)^2  \sqcup   \left( U^2  \times \widehat{N} \right)
 \end{equation}
 where $U=\KK^\times/\langle q\rangle$, $N=\ker(\chi)$. In particular, all finite dimensional simple left $H$-modules are either $1$-dimensional or $n$-dimensional, which coincides with the results obtained in \cite{Wang-Wu-Tan}.
 
 \medskip
 
 The general case for $\eta\neq \chi^t$ seems to be more complicated. Recall that for $N=\ker(\chi)\cap \ker(\eta)$ we have that $G/N$ embeds as a subgroup of the direct product $G/\ker(\chi) \times G/\ker(\eta)$. Assuming that $\chi(c) = q = \eta(b)^{-1}$ is an $n$th root of unity that generates the images of both characters $\chi$ and $\eta$, one has that $G/N$ is a subdirect product of $\ZZ_n \times \ZZ_n$. Consequently, $|G/N|$ divides $n^2$ and must itself have $n$ as a divisor since $b$ generates a subgroup of order $n$ in $G/N$. Thus, $|G/N| = nd$ for some divisor $d$ of $n$. By Proposition~\ref{proposition:torsionfree simples over general skew group ring}, any finite dimensional  $x$- and $y$-torsion-free simple left $H$-module has dimension $|G/N|=nd$. While for $\eta=\chi^t$, with $t$ relatively prime we have seen that $d=1$. However, the case $d=n$ is also possible. Let $r,s\in\ZZ$ be such that $\chi(b)=q^r$ and $\eta(c)=q^s$. 
 Consider the subset $\Lambda = \{b^i c^j N \mid 0 \leq i, j \leq n-1 \} \subseteq G/N$. Then $b^i c^j \in N$ if and only if 
 $\chi(b^i c^j) = 1 =\eta(b^i c^j)$, which is equivalent to   $q^{ri + j} = 1 = q^{sj - i}$. Hence $rsj = j \ (\text{mod } n)$. If $rs-1$ is relatively prime to $n$, then $j=i=0$. Thus $|G/N|=|\Lambda|=n^2$. A concrete example is as follows. Let $R=\KK[x,y]$ and $G$ the Klein group generated by $b$ and $c$. Let $q=-1$ and $\chi, \eta\in \widehat{G}$ determined by $\chi(c)=-1=\eta(b)$ and $\chi(b)=\eta(c)=1$. Then $n=2$, $N=\{1\}$ and there exists a $4$-dimensional simple module over $H(G,b,c,\chi,\eta,0)$.

\subsection{The Differential Operator Case}

Together with our general assumption of $\KK$ being an algebraically closed field of characteristic $0$, $G$ a finitely generated abelian group and  $\chi$ having finite order, say $n$, we will also assume  in this section $\beta(1 - bc) \neq 0$.  We saw in Remark \ref{remark:chi-eta-inverse} that then $\eta=\chi^{-1}$. Set $e:=c^{-1}-b$ and consider the Hopf algebra $H := H(G, \chi, \chi^{-1}, b, c, 1)$. We introduce some notation that will be used in this subsection. Let $\sigma=\tau_\chi: \KK[G] \to \KK[G]$ denote the left winding automorphism associated with $\chi$, defined by $\sigma(g) = \chi(g)g$ for all $g \in G$. Its inverse is given by $\sigma^{-1}(g) = \chi(g)^{-1}g$. For any element $u \in \KK[G]$ and integer $i \geq 1$, we use the notation:
\[ \wind{u}{i} := \sum_{k=0}^{i-1} \sigma^{-k}(u) = u + \sigma^{-1}(u) + \cdots + \sigma^{-(i-1)}(u). \]
We set $\wind{u}{0} = 0$. Note that if $g \in G$, this notation relates to our previous $q$-numbers by 
\[ \wind{g}{i} = \sum_{k=0}^{i-1} \chi^{-k}(g) g =  [i]_{\chi(g)^{-1}} g. \]
The global assumption  $\eta(b) = \chi(c)^{-1}=:q$ together with $\eta = \chi^{-1}$,  implies $\chi(b) = q^{-1}$ and 
\begin{equation}\label{eq:wind}
 \wind{e}{i} = [i]_{q^{-1}} c^{-1} - [i]_q b. 
 \end{equation}
Using this notation, the commutation relation from Equation~\eqref{equation:commutation relations 1} in Lemma \ref{lemma:change of variables in the nonzero derivation case} can be written compactly as:
\begin{equation}\label{equation:new commutation equation}
    yx^i - x^i y = \left([i]_q c^{-1} - [i]_{q^{-1}} b\right)x^{i-1} = x^{i-1} \wind{e}{i}
\end{equation}

\medskip

\subsubsection*{Torsion modules.} 
Extend $\rho \in \widehat{G}$ to an algebra homomorphism $\rho: \KK[G] \to \KK$, which will be denoted with the same symbol. To any $\rho \in \widehat{G}$, we  attach an infinite dimensional $H$-module $V(\rho)$ with basis $\{v_i \mid i \geq 0\}$ and with the $H$-module structure given by
\[x\cdot v_i = v_{i+1},\quad g\cdot v_i = \rho(\sigma^{-i}(g))v_i\]
for all $i \geq 0$ and $g \in G$ and
\[y\cdot v_i = \begin{cases}
    0 & \text{ if } i = 0\\ \rho(\wind{e}{i})v_{i-1} & \text{ if } i \geq 1
\end{cases}\]
It is routine to verify that this does define an $H$-module structure. For example, 
\begin{align*}
    (yx - xy - e)v_i & = yv_{i+1} - x\rho(\wind{e}{i})v_{i-1} - \rho(\sigma^{-i}(e))v_{i}\\
    & = \left(\rho(\wind{e}{i+1}) - \rho(\wind{e}{i})- \rho(\sigma^{-i}(e))\right)v_i = 0
\end{align*}
because $\wind{e}{i+1} = \wind{e}{i} + \sigma^{-i}(e)$, for $i>0$.

Recall that $\chi=\eta^{-1}$ has order $n$. Hence $\wind{e}{n}=0$. Let $d$ be the least positive integer such that $\wind{e}{d}=0$ and consider the factor module 
\[ \overline{V}(\rho): = V(\rho)/Hv_d, \]
 which is a $d$-dimensional $\KK$-space with basis $\{v_0, \ldots, v_{d-1}\}$ (where we still denote the basis element $v_i + Hv_d$ by $v_i$). 
\begin{proposition}
$\overline{V}(\rho)$ is  an $x$- and $y$-torsion simple  $H$-module.
\end{proposition}
\begin{proof}
Let $W$ be a nonzero submodule of $\overline{V}(\rho)$ and pick a  non-zero element $w = k_0 v_0 + \ldots + k_{d-1}v_{d-1} \in W$,  for some $k_0, \ldots, k_{d-1} \in \KK$. If $i$ is the least integer such that $k_i \neq 0$, then from the action of $x$ we see that $x^{d-i-1}k_{i}^{-1}w = v_{d-1}$ in $\overline{V}(\rho)$. Hence $v_{d-1} \in W$. From the action of $y$, we have
    \[y^j v_{d-1} = \left(\prod_{k=1}^{j}\rho(\wind{e}{d-k}) \right)v_{d-1-j}\]
    for all $1 \leq j \leq d-1$. Since $\KK$ is a field and by the minimality of $d$, none of the terms in this product is zero and therefore $W$ contains every basis element $v_i,\ 0 \leq i \leq d-1$. That is, $W = \overline{V}(\rho)$ and $\overline{V}(\rho)$ is simple. 
\end{proof}
\begin{proposition} The function $\widehat{G} \to \Irr(H)$ given by $\rho \mapsto \overline{V}(\rho)$ is injective.
\end{proposition}
\begin{proof}
Let $\rho_1, \rho_2\in \widehat{G}$ and set $V_1=\overline{V}(\rho_1)$ and $V_2=\overline{V}(\rho_2)$.  Suppose  $\varphi: V_1 \to V_2$ is an $H$-module isomorphism. Then the dimensions of $V_1$ and $V_2$ are equal and there exists $d\geq 1$ that is the least positive integer such that $\rho_i(\wind{e}{d})=0$, for $i=1,2$.  Suppose that the sets $\{v_{0}, v_{1}, \ldots, v_{d-1}\}$ and $\{w_{0}, w_{1}, \ldots, w_{d-1}\}$ are bases for $V_{1}$ and $V_{2}$, respectively, with their $H$-action as defined above.    Then $\varphi(v_0) = \sum_{i=0}^{d-1} k_i w_i$, for some scalars $k_i$, not all zero. Then for any $g\in G$:
\begin{equation*}
\sum_{i=0}^{d-1} \rho_1(g) k_i w_i =  \varphi(g v_0) = g \varphi(v_0) = \sum_{i=0}^{d-1} k_i \rho_2(\sigma^{-i}(g)) w_i = \sum_{i=0}^{d-1} k_i \chi^{-i}(g) \rho_2(g) w_i\end{equation*}
In particular, $\rho_1(g) = \rho_2(g)\chi^{-i}(g)$, for any $i$ with $k_i\neq 0$. Since $\chi$ has order $n$, there exists exactly one index $i$ such that $k_i\neq 0$. Hence, $\varphi(v_0)=k_iw_i$ and $\rho_1(g)=\rho_2(g)\chi^{-i}(g)$, for any $g\in G$. However, if $i>0$, then $\varphi(v_{d-1})=x^{d-1} \varphi(v_0) = k_i x^{d-1} w_i = 0$, which contradicts that $\varphi$ is injective. Hence $i=0$ and $\rho_1=\rho_2$. \end{proof}

\begin{proposition}
Every finite dimensional simple $x$- and $y$-torsion $H$-module is isomorphic to $\overline{V}(\rho)$ for some character $\rho\in \widehat{G}$.
\end{proposition}

\begin{proof}
Let $V$ be a finite dimensional simple $x$- and $y$-torsion $H$-module. Let $\{g_1,\ldots, g_m\}$ be a finite generating set of $G$.
Since the operators $g_i$ commute and $\KK$ is algebraically closed, they have
a common eigenvector $0\neq v\in V$. Since $V$ is $y$-torsion, there exists a least
$a\ge 1$ such that $y^a v=0$, and then $y^{a-1}v$ is again a common eigenvector for the
$g_i$ and is killed by $y$. Replacing $v$ by $y^{a-1}v$, we may assume $yv=0$.
Since $V$ is simple, $Hv=V$. Moreover, as $v$ is a common eigenvector for the group generators $g_i$, there are non-zero scalars $\rho(g_i)$, such that $g_iv = \rho(g_i)v$, for any $i$. Hence $\rho$ defines a character in $\widehat{G}$, which we identify with its algebra character $\rho:\KK[G]\to\KK$. 

%%%%

Let $m\geq 1$ be the least integer such that $x^mv = 0$. Since $\chi$ has order $n$, the element $x^{n}$ is central in $H$ and acts on $V$ by scalar multiplication, say $x^{n}v = \mu v$ for some $\mu \in \mathbb{K}$. But then from $x^{nm}v = \mu^{m}v = 0$, it follows that $\mu = 0$. In particular, this implies that $m \leq n$. 

On the other hand, let $d$ be the least positive integer such that $\rho([e]_{i}^{\sigma})=0$, which exists since $\rho([e]_{n}^{\sigma}) = 0$. In particular, $d \leq n$. 
If $x^{d}v \neq  0$ then since $V$ is simple,  $V = Hx^{d}v$. However, for some $g \in G$, the elements $v$ and $x^{d} v, \ldots, x^{n-1} v$ are eigenvectors with different eigenvalues and hence linearly independent, implying $v \not\in H x^{d} v = \bigoplus_{i=d}^{n-1} \KK x^iv$, which is a contradiction. Thus $x^dv=0$ and $d\leq m$. Since $x^{m - 1}v \neq 0$, from 
\[
0 = y x^{m} v = x^{m-1} [e]_{m}^\sigma v = \rho([e]_{m}^\sigma) x^{m-1} v
\]
we have $\rho([e]_{m}^{\sigma}) = 0$ and so $d = m$.  Hence, there is a map $\phi: V(\rho) \to V$ given by $v_i \mapsto x^i v$ with  $\phi(v_d) = x^dv=0$, which induces an $H$-module isomorphism $\overline{V}(\rho) \simeq V$. 
\iffalse
Then $\varphi$ is $x$-linear by construction and
$G$-linear because 
$g\cdot x^i v = \rho(\sigma^{-i}(g))x^i v$.
Moreover, for $i\ge 1$,
\[
y\cdot \varphi(v_i)=yx^i v=(x^i y + x^{i-1}\wind{e}{i})v
=\rho(\wind{e}{i})x^{i-1}v=\varphi(\rho(\wind{e}{i})v_{i-1})
=\varphi(y\cdot v_i),
\]
and also $y\cdot \varphi(v_0)=yv=0=\varphi(y\cdot v_0)$.
Thus $\varphi$ is an $H$-module homomorphism.
\fi
Since $\overline{V}(\rho)$ and $V$ are simple, $\varphi$ is an isomorphism.
\end{proof}

\begin{example} Let $q$ be an $n$th root of unity. Consider $G=\ZZ^2=\langle b,c\rangle$, where $b,c$ are generators of $G$. Define the character $\chi\in \widehat{G}$ by $\chi(b)=\chi(c)=q^{-1}$ and consider the Hopf algebra $H=H(G,b,c,\chi,\chi^{-1},1)$. Then $H$ is generated as an algebra by $x,y,b^{\pm 1},c^{\pm 1}$, subject to
$$ xy=yx + c^{-1}-b, \qquad bx = qxb, \qquad cx=qxc, \qquad yb=qby, \qquad yc=qcy,$$
where $b,c$ and group-like, $x$ is $(1,b)$-primitive and $y$ is $(c^{-1},1)$-primitive.

For any $1\leq d \leq n$ define a character $\rho \in \widehat{G}$ by $\rho(c)=1$ and $\rho(b)=q^{-d}$. Then $\overline{V}(\rho)$ is a $d$-dimensional simple torsion module over $H$.
\end{example}

\subsubsection*{Torsion-free modules.} We now characterize finite dimensional torsion-free simple $H$-modules, beginning with the $x$-torsion-free modules. The classification of $y$-torsion-free modules is analogous to the $x$-torsion-free case due to the symmetry of the construction. We already remarked earlier that $x$-torsion-free finite dimensional simple modules exist only if $\chi$ has finite order, which is our global assumption.% and in the rest of this section we will assume that $\chi$ has finite order $n$.

The classification of $x$-torsion-free simple left $H$-modules is similar
to that of the $x$- and $y$-torsion modules constructed above, with
minor but essential differences. Consider the $n$-dimensional vector space $V$ with basis $\{v_0, v_1, \ldots, v_{n-1}\}$, where $n$ is the order of $\chi$.  Let  $\rho\in \widehat{G}$ and  $\lambda, \mu \in \KK$ with $\lambda \neq 0$. We define an $H$-module structure on $V$ by the rules
\begin{eqnarray*}
x\cdot v_i &=& \begin{cases}
    v_{i+1} & \text{ if } 0 \leq i \leq n-2\\ \lambda v_0 & \text{ if } i = n-1
\end{cases}\\
y\cdot v_i &=& \begin{cases}
    \lambda^{-1}\mu v_{n-1} & \text{ if } i = 0\\ (\mu + \rho(\wind{e}{i}))v_{i-1} & \text{ if } 1 \leq i \leq n-1
\end{cases}
\end{eqnarray*}
and
\[g\cdot v_i = \rho(\sigma^{-i}(g)) v_i\]
for all $g\in G$ and $0 \leq i \leq n-1$. Indeed, one checks $(yx-xy)\cdot v_i = \rho(\sigma^{-i}(e))v_i$, 
 for all $i$. We will denote this $H$-module by $V_x (\rho, \lambda, \mu)$. 
\begin{proposition}
    The module $V_x (\rho, \lambda, \mu)$ is simple.
\end{proposition}
\begin{proof}
We use the same argument as in the torsion case with minor tweaks. Let $W$ be a nonzero submodule of $V_x (\rho, \lambda, \mu)$ and let $w = \sum_{i=0}^{n-1}k_i v_i$ be an element of $W$ of support. We claim that the support of $w$ is a singleton. Suppose this is not true and let $i_0 < i_1$ be distinct elements in the support of $w$. 
    %Then
    %
    %\[gw = \sum_{i=0}^{n-1}\rho(\sigma^{-i}(g))k_i v_i\]
    %
    %for all $g \in G$. 
Since $\chi$ has order  $n$ and $0<i_1-i_0<n$, there exists $h\in G$ such that $\chi^{i_1-i_0}(h)\neq1$. For such an $h$, the element 
    %
    %\begin{align*}
      %  \rho(h)w - \rho(\sigma^{i_1}(h))hw & = \sum_{i=0}^{n-1}\left(\rho(h) - \rho(\sigma^{i_1 - i}(h))\right)k_i v_i \\
      %  & = \sum_{i\neq i_1}\rho(h)(1 - \chi^{i_1 - i}(h))k_i v_i
    %\end{align*}
    \begin{equation*}
      \left(\rho(h) - \rho(\sigma^{i_1}(h))h\right)w  = \sum_{i=0}^{n-1}\left(\rho(h) - \rho(\sigma^{i_1 - i}(h))\right)k_i v_i   = \sum_{i\neq i_1}\rho(h)(1 - \chi^{i_1 - i}(h))k_i v_i
    \end{equation*}
    is certainly an element of $W$ whose support is smaller than that of $w$, contradicting the choice of $w$. Hence, the support of $w$ is a singleton. From the action of $x$, $W$ contains every basis element $v_i$ and so $W = V_x (\rho, \lambda, \mu)$ and $V_x (\rho, \lambda, \mu)$ is simple.
\end{proof}
The classification of these modules up to isomorphism will be addressed next. 
%First we record the following easy fact without proof:
%
%\begin{lemma}\label{lemma:auxilliary equation}
%$\wind{e}{n-k} + \sigma^{-(n-k)}(\wind{e}{k}) = 0$.
%\end{lemma}
%
\begin{proposition}\label{prop:iso-V-lambda-mu-rho}
$V_x(\rho_1, \lambda_1,\mu_1) \simeq V_x (\rho_2, \lambda_2,\mu_2)$ as left $H$-modules if and only if there exists $0\leq k <n$ such that 
    $$\lambda_1=\lambda_2, \qquad \rho_1=\rho_2\chi^{-k},  \qquad \mu_1 = \mu_2 + \rho_2(\wind{e}{k}).$$
\end{proposition}
\begin{proof}
($\Rightarrow$) 
Set $V_1=V_x(\rho_1, \lambda_1,\mu_1)$ and $V_2=V_x (\rho_2, \lambda_2,\mu_2)$ and suppose  $\varphi:V_1\to V_2$ is an $H$-module isomorphism. Let $\{v_0,\ldots, v_{n-1}\}$ be a basis of $V_1$ and $\{w_0, \ldots, w_{n-1}\}$ be a basis of $V_2$. Consider 
\[
\varphi(v_0) = \sum_{i=0}^{n-1} k_i w_i.
\]
For any $g\in G$ we obtain
\[
\sum_{i=0}^{n-1} k_i \rho_1(g) w_i = \varphi(gv_0) = g\varphi(v_0) = \sum_{i=0}^{n-1} k_i gw_i 
= \sum_{i=0}^{n-1} k_i \rho_2(\sigma^{-i}(g)) w_i 
\]
Thus, for any $i$ with $k_i\neq 0$ we conclude $\rho_1(g) = \rho_2(\sigma^{-i}(g)) = \rho_2(g)\chi^{-i}(g)$. Since $\chi$ has order $n$ and since $\varphi$ is injective, there exists precisely one index $i$ such that $k_i\neq 0$ and we conclude $\rho_1 = \rho_2 \chi^{-i}$ as well as $\varphi(v_0)= k_i w_i$. Furthermore, 
\[ \lambda_1 k_i w_i = \varphi(x^n v_0) = x^n \varphi(v_0) =k_i \lambda_2 w_i.\]
Hence $\lambda_1=\lambda_2 =: \lambda $.  For $i>0$ we have 
\begin{eqnarray*}
 y\varphi(v_0) &=& k_i yw_i = k_i (\mu_2 + \rho(\wind{e}{i}))w_{i-1},\\
 \varphi(yv_0) &=& \lambda^{-1}\mu_1 \varphi(v_{n-1})  = \lambda^{-1}\mu_1 k_i x^{n-1} w_i   = k_i \mu_1 w_{i-1} .
 \end{eqnarray*}
Thus $\mu_1 = \mu_2 + \rho(\wind{e}{i})$. For $i=0$ we have 
\begin{eqnarray*}
 y\varphi(v_0) &=& k_0 y w_0 = k_0\lambda^{-1}\mu_2 w_{n-1} \\
 \varphi(yv_0) &=& \lambda^{-1}\mu_1 \varphi(v_{n-1})  = \lambda^{-1}\mu_1 k_0 w_{n-1}.
 \end{eqnarray*}
Hence $\mu_1 = \mu_2 = \mu_2 + \rho(\wind{e}{0})$.

($\Leftarrow$) Conversely, suppose $\lambda_1=\lambda_2$ and that there exists $0\leq k < n$ such that  $\rho_1=\rho_2\chi^{-k} $ and $\mu_1 = \mu_2 + \rho_2(\wind{e}{k})$ hold. 
Define a linear map $\varphi: V_1 \to V_2$ by $\varphi(v_{i}) = x^i w_k$, for $0\leq i<n$. Then
\[ \varphi(xv_i) = \varphi(v_{i+1}) = x^{i+1}w_k = x \varphi(v_i),\]  for $i<n-1$,  while  $\varphi(xv_{n-1})=\lambda \varphi(v_0)= \lambda w_k = x x^{n-1} w_k= x\varphi(v_{n-1}).$

For  $g \in G$, we have, using $\rho_1 = \rho_2 \chi^{-k} = \rho_2\circ \sigma^{-k}$:
\[\varphi(g v_{i}) = \rho_1 (\sigma^{-i}(g)) \varphi(v_{i}) = \rho_1 (\sigma^{-i}(g)) x^iw_k = \rho_2 (\sigma^{-i-k}(g)) x^iw_k\]

while
\[g \varphi(v_{i}) = g x^i w_k = \chi^{-i}(g) x^i \rho_2(\sigma^{-k}(g)) w_k  =\rho_2(\sigma^{-i-k}(g)) \varphi(v_i).\]
Thus $\varphi$ is $G$-linear. To check the $y$-linearity, note
\[
\varphi(yv_0)=\lambda^{-1}\mu_1\varphi(v_{n-1})
=\mu_1\lambda^{-1}x^{n-1}w_k 
=(\mu_2 + \rho_2(\wind{e}{k})w_{k-1} = y w_k = y\varphi(v_0).
\]

Finally, for $i \geq 1$, we have
\[\varphi(y v_{i}) = \varphi((\mu_1 + \rho_1 (\wind{e}{i}))v_{i-1}) = (\mu_1 + \rho_1 (\wind{e}{i})) x^{i-1}w_{k}\]
while
\[y \varphi(v_{i}) = yx^i w_k = (x^iy + x^{i-1}\wind{e}{i})w_k = (\mu_2 + \rho_2 (\wind{e}{i+k})) x^{i-1} w_{k}\]
Since $\mu_1 = \mu_2 + \rho_2(\wind{e}{k})$ and $\rho_1=\rho_2\circ \sigma^{-k}$, we obtain
\[ \mu_1 + \rho_1 (\wind{e}{i}) = \mu_2 + \rho_2(\wind{e}{k} + \sigma^{-k} (\wind{e}{i})) = \mu_2 + \rho_2 (\wind{e}{i+k}).\]
Thus $\varphi$ is a non-zero homomorphism of $H$-modules and hence an isomorphism as $V_1$ and $V_2$ are simple.
\end{proof}
We conclude by saying that these are the only $x$-torsion-free simple $H$-modules.
\begin{proposition}
    Every $x$-torsion-free finite dimensional simple left $H$-module is isomorphic to $V_x (\rho, \lambda, \mu)$ for some $\lambda, \mu \in \KK$ with $\lambda \neq 0$, and  $\rho\in \widehat{G}$.
\end{proposition}
\begin{proof}
    Let $V$ be a finite dimensional simple $H$-module that is $x$-torsion-free. The elements $x^n, xy$, the $g_i$ ($1 \leq i \leq d$) act on $V$ as commuting linear operators and since $\KK$ is algebraically closed, they must have a common eigenvector, say $v$, in $V$. Let $x^n \cdot v = \lambda v$, $xy\cdot v = \mu v$, and $g \cdot v = \rho(g)v$ for some $\lambda, \mu \in \KK$ with $\lambda \neq 0$ and $\rho\in \widehat{G}$.

    The elements $v_i := x^i v$ for $0 \leq i \leq n-1$ are eigenvectors under the action of some $g \in G$ with distinct eigenvalues and so they are linearly independent. It is easy to see that 
    \[x\cdot v_{n-1} = \lambda v_0\quad \text{and}\quad x \cdot v_i = v_{i+1}\quad \text{for}\quad 0 \leq i \leq n-2\]
    and that
    \[g\cdot v_i = \rho(\sigma^{-i}(g))v_i\]
    for all $g \in G$.
    From
    \[\lambda y v_0 = x^n y v = x^{n-1}xyv = x^{n-1}\mu v = \mu v_{n-1}\]
    we also see that $y\cdot v_0 = \lambda^{-1}\mu v_{n-1}$. Moreover, for $1 \leq i \leq n-1$ we have from \eqref{equation:new commutation equation}
    \[ y\cdot v_i = yx^i v_0 = (x^i y + x^{i-1}\wind{e}{i})v_0 = (\mu + \rho(\wind{e}{i}))v_{i-1}\]
    So, the sum $\sum_{i=0}^{n-1}\KK v_i$ is direct and is a nonzero submodule of $V$ with the above $H$-module structure. Since $V$ is simple, we must have $V = \sum_{i=0}^{n-1}\KK v_i$ and that $V \simeq V_x (\rho, \lambda, \mu)$.
\end{proof}
The above construction of $x$-torsion-free modules can be carried out with minor changes to describe the $y$-torsion-free simple $H$-modules without a change. Therefore, we give the description of these modules without proof.

By symmetry with the $x$-torsion-free case, the construction of $y$-torsion-free simple $H$-modules is governed by the way $y$ commutes with $x$ and the group algebra $\KK[G]$. Consequently, one expects the left winding automorphism associated with the character $\eta$ to appear in the description of such modules. In the differential operator case considered here, we have $\eta=\chi^{-1}$, and therefore the corresponding winding automorphism is $\sigma^{-1}$.

With this in mind, let $V$ be an $n$-dimensional $\KK$-vector space with basis $\{v_0,v_1,\ldots,v_{n-1}\}$, let $\rho:\KK[G]\to\KK$ be a character, and let $\lambda,\mu\in\KK$ with $\mu\neq 0$. We define an $H$-module structure on $V$ by the rules
\[y\cdot v_i = \begin{cases}v_{i+1} & \text{ if } 0 \leq i \leq n-2\\ \mu v_0 & \text{ if } i = n-1\end{cases}\]
\[x\cdot v_i = \begin{cases}\mu^{-1}\lambda v_{n-1} & \text{ if } i = 0\\ (\lambda - \rho(\wind{e}{i}))v_{i-1} & \text{ if } 1 \leq i \leq n-1
\end{cases}\]
and
\[g\cdot v_i = \rho(\sigma^{i}(g))v_i\]
for all $g \in G$ and $0 \leq i \leq n-1$. We will denote this $H$-module by $V_y (\rho, \lambda, \mu)$.
\begin{proposition}
    Retain the above notation.
    \begin{enumerate}
        \item The module $V_y (\rho, \lambda, \mu)$ is simple.
        \item Two such modules $V_y (\rho_1, \lambda_1, \mu_1)$ and $V_y (\rho_2, \lambda_2, \mu_2)$ are isomorphic if and only if 
        $$\mu_1 = \mu_2, \qquad \rho_2 = \rho_1 \circ \sigma^{-k}, \qquad \lambda_1 = \lambda_2 - \rho_{2}(\wind{e}{k}),$$
        for some $0\leq k < n$.
        \item Every $y$-torsion-free finite dimensional simple left $H$-module is isomorphic to $V_y (\rho, \lambda, \mu)$ for some $\lambda, \mu \in \KK$ with $\mu \neq 0$, and  $\rho\in \widehat{G}$.
    \end{enumerate}
\end{proposition}
The duality between the $x$- and $y$-torsion-free modules can alternatively be seen from the symmetry between the roles of $x$ and $y$ in the construction of $H$. That is, if we interchange the roles of $x$ and $y$ and write $H$ as $H(G, \chi^{-1}, \chi, c^{-1}, b^{-1}, 1)$, we get a Hopf algebra that is isomorphic to $H(G, \chi, \chi^{-1}, b, c, 1)$. 
%%%%%
\section{Tensor Products of Simple Modules in the Skew Group Case}
In this section, we calculate the tensor products of simple modules in the skew group ring case and give multiplication rules in the representation ring of $H$. Recall that, as an abelian group, the representation ring of $H$ is generated by the isomorphism classes of finite-dimensional simple modules, modulo the relations
\[[V] = [V_1] + [V_2] + \ldots + [V_r]\]
whenever $V$ has a composition series with composition factors $V_1, \ldots, V_r$. Multiplication is induced by the tensor products 
\[[V]\cdot [W] = [V \otimes W].\]

We continue with our assumptions that $\KK$ is an algebraically closed field of characteristic zero, $G$ is a finitely generated abelian group, and the characters $\chi$ and $\eta$ have finite order.

Under our running assumption that $\eta = \chi^t$ for some $t$ relatively prime to the order $n$ of $\chi$, we have $\ker(\chi) = \ker(\eta)$ and thus $N = \ker(\chi) = \ker(\eta)$. Moreover, the group $G/N \subset \KK^\times$ is cyclic and has order $n$, say $G/N \simeq \langle aN\rangle$ for some $a \in G$. 

Recall from Section~\ref{section:finite dimensional simple $H$-modules} that every finite dimensional $x$- and $y$-torsion simple $H$-module $V$ is one-dimensional where the action of $g \in G$ on any $v \in V$ is given by $g\cdot v = \rho(g)v$ for some character $\rho \in \widehat{G}$. We denoted this module by $\KK_\rho$. On the other hand, an $x$- or $y$-torsion-free simple left $H$-module is $n$-dimensional with basis $\{v_0, v_1, \ldots, v_{n-1}\}$ where the action of $H$ is given by
\[x\cdot v_i = \chi(a^i)\lambda v_i, \quad y\cdot v_i = \eta(a^i)\mu v_i, \quad h\cdot v_i = \rho(h)v_i\]
and 
\[a\cdot v_i = v_{i+1(\text{mod }n)}\]
for all $i = 0, 1, \ldots, n-1$ and $h \in N$ where $\rho\in \widehat{N}$, $\lambda, \mu \in \KK$ are scalars at least one of which is nonzero. We will denote this module by $V(\lambda, \mu, \rho)$. 

We study the tensor products of simple $H$-modules case by case.

\subsection{Torsion vs Torsion}
Let $\rho$ and $\sigma$ be two characters in $\widehat{G}$ and consider the torsion simple $H$-modules $\mathbb{K}_\rho$ and $\mathbb{K}_\sigma$ with basis $\{v\}$ and $\{w\}$ respectively. For all $g \in G$, we have
\[g\cdot (v\otimes w) = (gv) \otimes (gw) = \rho(g)\sigma(g)v\otimes w\]
Also,
\[x\cdot (v\otimes w) = (x\otimes 1 + b \otimes x)(v\otimes w) = xv\otimes w + bv\otimes xw = 0\]
Similarly, $y\cdot (v\otimes w) = 0$ so the tensor product is again torsion and $\mathbb{K}_\rho \otimes \mathbb{K}_\sigma \simeq \mathbb{K}_{\rho\sigma}$.

Thus, in the representation ring, we have $[\mathbb{K}_\rho]\cdot [\mathbb{K}_\sigma] = [\mathbb{K}_{\rho\sigma}] = [\mathbb{K}_\sigma]\cdot [\mathbb{K}_\rho]$. In particular, the subring $\mathbb{Z}[\widehat{G}]$ generated by the torsion simple $H$-modules embeds as a subring in the representation ring of $H$. 

\subsection{Torsion vs Torsion-free}
Let $\rho \in \widehat{G}, \sigma\in\widehat{N}$ be two characters. Let $\lambda, \mu \in \mathbb{K}$ be two scalars, not both zero. Consider the torsion simple $H$-module $\mathbb{K}_\rho$ with basis $\{u\}$ and the torsion-free simple $H$-module $V (\lambda, \mu, \sigma)$ with its standard basis $\{v_i \mid 0 \leq i \leq n-1 \}$. 

Since $x$ acts as zero on $\mathbb{K}_\rho$, we have for every $0 \leq i \leq n-1$ that
\[x\cdot (u \otimes v_i) = xu \otimes v_i + bu \otimes xv_i = \rho(b)u\otimes xv_i\]
So, on the tensor product, the action of $x$ is modified by a factor of $\rho(b)$. Similarly, for every $0 \leq i \leq n-1$, we have
\[y\cdot (u \otimes v_i) = yu \otimes c^{-1}v_i + u \otimes yv_i = u \otimes yv_i\]
Finally, since every $g \in G$ can be written as $g = a^j h$ for some $h \in N$ and $0 \leq j \leq n-1$, we have
\[g \cdot (u \otimes v_i) = a^j h \cdot (u \otimes v_i) = a^j hu \otimes a^j hv_i = \rho(h)\rho(a^j) \sigma(h)(u \otimes v_{i+j}).\]
Hence, after a change of the basis elements $u \otimes v_i \mapsto u \otimes \rho(a^i)v_i$ in the tensor product, we see that in the representation ring we have
\[[\mathbb{K}_\rho]\cdot [V(\lambda, \mu, \sigma)] = [V(\rho(b)\lambda, \mu, (\rho|_N) \sigma)].\]

If we now calculate the right tensor product $V (\lambda, \mu, \sigma) \otimes \mathbb{K}_\rho$, we see that there is a symmetry between the actions of $x$ and $y$ in the two cases. From 
\[x\cdot (v_i \otimes u) = xv_i \otimes u + bv_i \otimes xu = xv_i \otimes u\]
and
\[y\cdot (v_i \otimes u) = yv_i \otimes c^{-1}u + v_i \otimes yu = \rho(c^{-1})yv_i \otimes u\]
we see that the action of $x$ does not change and the action of $y$ is modified by $\rho(c^{-1})$. Similar to the previous case, for an element $g = a^j h \in G$, we have
\[a^j h \cdot v_i \otimes u=\rho(h)\rho(a^j)\sigma(h)v_{i+j} \otimes u\]
and again after a change of the basis elements $v_i \otimes u \mapsto \rho(a^i)v_i \otimes u$, it follows that
\[[V(\lambda, \mu, \sigma)]\cdot [\mathbb{K}_\rho] = [V(\lambda, \rho(c^{-1})\mu, (\rho|_N)\sigma)].\]
\subsection{Torsion-free vs Torsion-free.}
Let $\rho_1, \rho_2 \in \widehat{N}$ be characters and $\lambda_1, \lambda_2, \mu_1, \mu_2$ be nonzero scalars. Consider the torsion-free simple $H$-modules $V(\lambda_1,  \mu_1, \rho_1)$ and $V(\lambda_2, \mu_2, \rho_2)$ with their standard bases $\{v_0, v_1, \ldots, v_{n-1}\}$ and $\{w_0, w_1, \ldots, w_{n-1}\}$, respectively. Let $V := V(\lambda_1, \mu_1, \rho_1) \otimes V(\lambda_2, \mu_2, \rho_2)$. Let $X$ and $Y$ denote the operators defined by the actions of $x$ and $y$ on the tensor product, respectively. Since $\Delta(x^n) = x^n \otimes 1 + b^n \otimes x^n$ and $\Delta(y^n) = y^n \otimes c^{-n} + 1 \otimes y^n$, $X^n$ acts on the tensor product by multiplication by the scalar $L = \lambda_{1}^{n} + \rho_1(h_{1}^n)\lambda_{2}^{n}$ and $Y^n$ acts as multiplication by $M = \rho_2 (h_{2}^{-n})\mu_{1}^{n} + \mu_{2}^n$.

For $j = 0, 1, \ldots, n-1$, consider the subspaces 
\[U_j := \text{span}\{v_i \otimes w_j \mid 0 \le i \le n-1\}\]
and
\[W_j := \text{span}\{v_j \otimes w_i \mid 0 \le i \le n-1\}.\]
Then each $U_j$ is stable under the action of $X$ and each $W_j$ is stable under the action of $Y$. The tensor product $V$ decomposes into the direct sum 
\[V = \oplus_{j=0}^{n-1} U_j = \oplus_{j=0}^{n-1}W_j.\]

(i) Suppose that $L$ and $M$ are nonzero. $X^n$ acts as multiplication by $L$ on the tensor product and therefore on each $U_k$. Hence, the minimal polynomial of $X|_{U_k}$ divides $t^n - L$. Since $L$ is nonzero and $\mathbb{K}$ is algebraically closed with characteristic zero, the polynomial $t^n - L$ has $n$ distinct roots in $\mathbb{K}$. Hence $X|_{U_k}$ is diagonalizable for every $k$, and therefore $X$ is diagonalizable on $V$, with the eigenvalues are equal to some $n$th roots of $L$. Similarly, $Y$ is diagonalizable on $V$, with eigenvalues being equal to some $n$th roots of $M$.

Let us write $\xi:= \chi(a)$. From $xa = \xi ax$, it follows that if $v$ is an eigenvector of $X$ with eigenvalue $l$, then 
\[X(av) = \xi lav\]
so the action of $a$ sends $l$-eigenspaces isomorphically onto $\xi l$-eigenspaces. Therefore, all the $n$th roots $l, \xi l, \ldots, \xi^{n-1}l$ of $L$ occur with the same multiplicity. Since $\dim V = n^2$, each root occurs with multiplicity $n$. Similarly, since $ya = \xi^t ay$, if $Yv = mv$, then 
\[Y(av) = \xi^t mav\]
and since $t$ and $n$ are relatively prime, the eigenvalues 
\[m, \xi^t m, \xi^{2t}m, \ldots, \xi^{(n-1)t}m \]
are all $n$th roots of $M$, again each with multiplicity $n$.

Since $X$ and $Y$ commute, the tensor product $V$ has a basis of common eigenvectors of $X$ and $Y$. Let $v_1$ be a common eigenvector such that $Xv_1 = l v_1$ and $Yv_1 = m_1 v_1$. The vectors $v_1, av_1, \ldots, a^{n-1}v_1$ are linearly independent, as they have distinct eigenvalues since $\xi$ is a primitive $n$-th root of unity. Moreover, their span is the torsion-free simple $H$-module $V_1 = V(l, m_1, \rho_1\rho_2)$. Since each eigenvalue has multiplicity $n$, we can choose a common eigenvector $v_2$ that is not in $V_1$ with the same $X$-eigenvalue $xv_2 = l v_2$ and $yv_2 = m_2 v_2$, giving another torsion-free simple $H$-module $V_2 = V(l, m_2, \rho_1\rho_2)$. Continuing this way, we find out that the tensor product has a direct sum decomposition into torsion-free simples
\[V = \oplus_{i=1}^{n}V(l, m_i, \rho_1\rho_2)\]
where $l$ is an $n$th root of $L$ and the $m_i$ are some $n$th roots of $M$. This means that in the representation ring, we have
\[[V(\lambda_1, \mu_1, \rho_1)]\cdot [V(\lambda_2, \mu_2, \rho_2)] = \sum_{i=1}^{n}[V(l,m_i,\rho_1\rho_2)].\]

(ii) Now suppose that $L \neq 0$ and $M = 0$. In this case $Y$ acts as zero on every composition factor of $V$. Since $Y$ acts nilpotently on the eigenspace $E_l$, there is a $Y$-stable flag 
\[0 = F_0 \subset F_1 \subset \cdots \subset F_n = E_l\]
with $\dim F_i / F_{i-1} = 1$ and $YF_i \subset F_{i-1}$. The subspaces 
\[\widetilde{F_i} = \sum_{k=0}^{n-1}a^k F_i\]
form an $H$-stable filtration of $V$ where each quotient
\[\widetilde{F_i}/\widetilde{F}_{i-1}\]
is isomorphic to $V(l, 0, \rho_1 \rho_2)$. Hence, in the representation ring, we have
\[[V(\lambda_1, \mu_1, \rho_1)]\cdot [V(\lambda_2, \mu_2, \rho_2)] = n[V(l, 0, \rho_1\rho_2)]\]

(iii) The case in which $L = 0$ and $M \neq 0$ is symmetrical to (ii) and we have 
\[[V(\lambda_1, \mu_1, \rho_1)]\cdot[V(\lambda_2, \mu_2, \rho_2)] = n[V(0, m, \rho_1\rho_2)]\]
where $m$ is an $n$th root of $M$.

(iv) Finally, we consider the case where $M$ and $L$ are both zero. In this case, $X$ and $Y$ are nilpotent operators, and therefore every vector in the tensor product is $x$- and $y$-torsion. Consequently, composition factors are all one-dimensional torsion modules. We analyze which simples $\mathbb{K}_\sigma$ can appear as a composition factor. 

The tensor product has a subspace decomposition
\[V = \oplus_{k=0}^{n-1}W_k\]
where 
\[W_k = \mathrm{span}\{v_i \otimes w_{i+k} \mid i = 0,1, \ldots, n-1\}.\]
We have
\[a\cdot v_i \otimes w_{i+k} = v_{i+1}\otimes w_{i+1+k}\]
and so each subspace $W_k$ is stable under the action of $a$ (indices calculated modulo $n$). The action of $a$ on $W_k$ is a cyclic shift, and as such its eigenvalues are the $n$-th roots of unity. Since $\xi = \chi(a)$ is a primitive $n$-th root of unity, these eigenvalues are precisely $1, \xi, \ldots, \xi^{n-1}$.

Let $\sigma \in \widehat{G}$ and consider a one-dimensional torsion simple module $\mathbb{K}_\sigma$ which appears as a composition factor of the tensor product. Since every $g \in N$ acts on the tensor product as $\rho_1 \rho_2(g)$, we see that the restriction of $\sigma$ to $N$ is $\rho_1 \rho_2$. Since every $g \in G$ can be written as $g = ha^i$ for some $h \in N$ and $0\le i \le n-1$, the character $\sigma$ must be of the form $\sigma(ha^i) = \rho_1 \rho_2(h)\sigma(a)^i$. From the above paragraph, we know that $\sigma(a) = \chi(a)^i = \xi^i$ for some $i$. 

Therefore, we have 
\[[V(\lambda_1, \mu_1, \rho_1)]\cdot[V(\lambda_2, \mu_2, \rho_2)] = \sum_{i = 0}^{n-1}n\left[\mathbb{K}_{(\rho_1 \rho_2, \xi^i)}\right]\]
in the representation ring, where $(\rho_1\rho_2, \xi^i)$ denotes the unique character in $G$ whose restriction to $N$ is $\rho_1\rho_2$ and which sends $a$ to $\xi^i$. 

We gather everything in the following theorem:
\begin{theorem}[Tensor product rules in the skew group ring case]
Assume that \(e=0\), that \(\chi\) has finite order \(n\), and that
\(\eta=\chi^{t}\) with \((t,n)=1\). Let
\[N=\ker(\chi)=\ker(\eta),\qquad G/N=\langle aN\rangle, \]
and write \(\xi=\chi(a)\). Suppose also that
\[b=a^r h_1,\qquad c=a^s h_2\]
with \(h_1,h_2\in N\).

\begin{enumerate}

\item \textbf{Torsion vs torsion.}
For \(\rho,\sigma\in\widehat{G}\),
\[[\mathbb K_\rho]\,[\mathbb K_\sigma] = [\mathbb K_{\rho\sigma}].\]

\item \textbf{Torsion vs torsion-free.}
Let \(\rho\in\widehat{G}\), \(\sigma\in\widehat{N}\), and
\(\lambda,\mu\in\mathbb K\), not both zero. Then
\[[\mathbb K_\rho]\,[V(\lambda,\mu,\sigma)] = [V(\rho(b)\lambda,\mu,(\rho|_N)\sigma)],\]
and
\[[V(\lambda,\mu,\sigma)]\,[\mathbb K_\rho] = [V(\lambda,\rho(c^{-1})\mu,(\rho|_N)\sigma)].\]

\item \textbf{Torsion-free vs torsion-free.}
Let
\[V_1=V(\lambda_1,\mu_1,\rho_1),\qquad V_2=V(\lambda_2,\mu_2,\rho_2)\]
be torsion-free simple modules with \(\rho_1,\rho_2\in\widehat{N}\) and
\(\lambda_i,\mu_i\in\mathbb K\), \((\lambda_i,\mu_i)\neq(0,0)\). Set
\[L=\lambda_1^n+\rho_1(h_1^n)\lambda_2^n, \qquad M=\rho_2(h_2^{-n})\mu_1^n+\mu_2^n.\]

\begin{enumerate}
    \item If \(L\neq 0\) and \(M\neq 0\), fix an \(n\)-th root \(\ell\) of \(L\).
    Let \(m_1,\ldots,m_n\) be the eigenvalues (with multiplicity) of
    \(Y|_{E_\ell}\), where \(E_\ell=\ker(X-\ell I)\). Then
    \[[V_1][V_2] = \sum_{i=1}^n [V(\ell,m_i,\rho_1\rho_2)].\]

    \item If \(L\neq 0\) and \(M=0\), then for any \(n\)-th root \(\ell\) of \(L\),
    \[[V_1][V_2] = n[V(\ell,0,\rho_1\rho_2)].\]

    \item If \(L=0\) and \(M\neq 0\), then for any \(n\)-th root \(m\) of \(M\),
    \[[V_1][V_2] = n[V(0,m,\rho_1\rho_2)].\]

    \item If \(L=0=M\), then
    \[    [V_1][V_2]  = \sum_{i=0}^{n-1} n\,[\mathbb K_{(\rho_1\rho_2,\xi^i)}].\]
\end{enumerate}
\end{enumerate}
\end{theorem}


\begin{bibdiv}
\begin{biblist} 

\bib{Beattie}{article}{
   author={Beattie, M.},
   author={D\u{a}sc\u{a}lescu, S.},
   author={Gr\"{u}nenfelder, L.},
   title={Constructing pointed Hopf algebras by Ore extensions},
   journal={J. Algebra},
   volume={225},
   date={2000},
   number={2},
   pages={743--770},
   issn={0021-8693},
   review={\MR{1741560}},
   doi={10.1006/jabr.1999.8148},
}

\bib{Brown-O'Hagan-Zhang}{article}{
   author={Brown, K. A.},
   author={O'Hagan, S.},
   author={Zhang, J. J.},
   author={Zhuang, G.},
   title={Connected Hopf algebras and iterated Ore extensions},
   journal={J. Pure Appl. Algebra},
   volume={219},
   date={2015},
   number={6},
   pages={2405--2433},
   issn={0022-4049},
   review={\MR{3299738}},
   doi={10.1016/j.jpaa.2014.09.007},
}

\bib{Fantino-Garcia}{article}{
   author={Fantino, Fernando},
   author={Garcia, Gaston Andr\'es},
   title={On pointed Hopf algebras over dihedral groups},
   journal={Pacific J. Math.},
   volume={252},
   date={2011},
   number={1},
   pages={69--91},
   issn={0030-8730},
   review={\MR{2862142}},
   doi={10.2140/pjm.2011.252.69},
}

\bib{Gelaki}{article}{
   author={Gelaki, Shlomo},
   title={Pointed Hopf algebras and Kaplansky's 10th conjecture},
   journal={J. Algebra},
   volume={209},
   date={1998},
   number={2},
   pages={635--657},
   issn={0021-8693},
   review={\MR{1659891}},
   doi={10.1006/jabr.1998.7513},
}

\bib{Goodearl}{article}{
   author={Goodearl, K. R.},
   title={Prime ideals in skew polynomial rings and quantized Weyl algebras},
   journal={J. Algebra},
   volume={150},
   date={1992},
   number={2},
   pages={324--377},
   issn={0021-8693},
   review={\MR{1176901}},
   doi={10.1016/S0021-8693(05)80036-5},
}

\bib{Halbig}{article}{
   author={Halbig, Sebastian},
   title={Generalized Taft algebras and pairs in involution},
   journal={Comm. Algebra},
   volume={49},
   date={2021},
   number={12},
   pages={5181--5195},
   issn={0092-7872},
   review={\MR{4328530}},
   doi={10.1080/00927872.2021.1939043},
}

\bib{Halbig-Krahmer}{article}{
   author={Halbig, Sebastian},
   author={Kr\"ahmer, Ulrich},
   title={A Hopf algebra without a modular pair in involution},
   conference={
      title={Geometric methods in physics XXXVII},
   },
   book={
      series={Trends Math.},
      publisher={Birkh\"auser/Springer, Cham},
   },
   isbn={978-3-030-34072-8},
   isbn={978-3-030-34071-1},
   date={[2019] \copyright 2019},
   pages={140--142},
   review={\MR{4143889}},
   doi={10.1007/978-3-030-34072-8\_14},
}


\bib{Kaplansky}{book}{
   author={Kaplansky, Irving},
   title={Bialgebras},
   series={Lecture Notes in Mathematics},
   publisher={University of Chicago, Department of Mathematics, Chicago, IL},
   date={1975},
   pages={iv+57},
   review={\MR{0435126}},
}

\bib{Kassel}{book}{
   author={Kassel, Christian},
   title={Quantum groups},
   series={Graduate Texts in Mathematics},
   volume={155},
   publisher={Springer-Verlag, New York},
   date={1995},
   pages={xii+531},
   isbn={0-387-94370-6},
   review={\MR{1321145}},
   doi={10.1007/978-1-4612-0783-2},
}

\bib{McConnell-Robson}{book}{
   author={McConnell, J. C.},
   author={Robson, J. C.},
   title={Noncommutative Noetherian rings},
   series={Pure and Applied Mathematics (New York)},
   note={With the cooperation of L. W. Small;
   A Wiley-Interscience Publication},
   publisher={John Wiley \& Sons, Ltd., Chichester},
   date={1987},
   pages={xvi+596},
   isbn={0-471-91550-5},
   review={\MR{0934572}},
}

\bib{Muller}{article}{
   author={M\"{u}ller, Eric},
   title={Finite subgroups of the quantum general linear group},
   journal={Proc. London Math. Soc. (3)},
   volume={81},
   date={2000},
   number={1},
   pages={190--210},
   issn={0024-6115},
   review={\MR{1757051}},
   doi={10.1112/S002461150001248X},
}


\bib{Panov}{article}{
   author={Panov, A. N.},
   title={Ore extensions of Hopf algebras},
   language={Russian, with Russian summary},
   journal={Mat. Zametki},
   volume={74},
   date={2003},
   number={3},
   pages={425--434},
   issn={0025-567X},
   translation={
      journal={Math. Notes},
      volume={74},
      date={2003},
      number={3-4},
      pages={401--410},
      issn={0001-4346},
   },
   review={\MR{2022506}},
   doi={10.1023/A:1026115004357},
}

\bib{Passman}{book}{
   author={Passman, Donald S.},
   title={Infinite crossed products},
   series={Pure and Applied Mathematics},
   volume={135},
   publisher={Academic Press, Inc., Boston, MA},
   date={1989},
   pages={xii+468},
   isbn={0-12-546390-1},
   review={\MR{0979094}},
}

\bib{Takeuchi}{article}{
   author={Takeuchi, Mitsuhiro},
   title={Representations of the Hopf algebra $U(n)$},
   conference={
      title={Hopf algebras and generalizations},
   },
   book={
      series={Contemp. Math.},
      volume={441},
      publisher={Amer. Math. Soc., Providence, RI},
   },
   isbn={978-0-8218-3820-4},
   date={2007},
   pages={155--174},
   review={\MR{2381540}},
   doi={10.1090/conm/441/08504},
}
\bib{Wang-Wu-Tan}{article}{
    author={Wang, Jing},
    author={Wu, Zhixiang},
    author={Tan, Yan},
    title={Some Hopf algebras related to $\mathfrak{sl}_2$},
    journal={Comm. Algebra},
    volume={49},
    date={2021},
    number={8},
    pages={3335--3368},
    issn={0092-7872},
    review={\MR{4283152}},
    doi={10.1080/00927872.2021.1894568},
}

\bib{Wu-Zhang}{article}{
   author={Wu, Q.-S.},
   author={Zhang, J. J.},
   title={Noetherian PI Hopf algebras are Gorenstein},
   journal={Trans. Amer. Math. Soc.},
   volume={355},
   date={2003},
   number={3},
   pages={1043--1066},
   issn={0002-9947},
   review={\MR{1938745}},
   doi={10.1090/S0002-9947-02-03106-9},
}
\end{biblist}
\end{bibdiv}
\end{document}